\documentclass[12pt,reqno]{article}

\usepackage[usenames]{color}
\usepackage{amssymb}
\usepackage{graphicx}
\usepackage{amscd}

\usepackage[colorlinks=true,
linkcolor=webgreen,
filecolor=webbrown,
citecolor=webgreen]{hyperref}

\definecolor{webgreen}{rgb}{0,.5,0}
\definecolor{webbrown}{rgb}{.6,0,0}

\usepackage{color}
\usepackage{fullpage}
\usepackage{graphicx,amsmath,amsthm,amssymb,latexsym,amsbsy}
\usepackage{pstricks}
\usepackage{epsf}

\newcommand{\seqnum}[1]{\href{http://www.research.att.com/\%7Enjas/sequences/#1}{\underline{#1}}}

\begin{document}

~\vskip 1in
%\begin{center}
%\epsfxsize=4in
%\leavevmode\epsffile{./figures/logo129.eps}
%\end{center}

\begin{center}
\vskip 1cm{\LARGE\bf Lens Sequences}

\vskip 1cm
\large
Jerzy Kocik\\
Mathematics Department \\
Southern Illinois University Carbondale \\
Carbondale, IL 62901, USA \\
\href{mailto:jkocik@math.siu.edu}{\tt jkocik@math.siu.edu}\\
\end{center}

\vskip .2in

\begin{abstract}  
A family of sequences produced by a non-homogeneous 
linear recurrence formula derived from the geometry of circles inscribed in 
lenses is introduced and studied. Mysterious ``underground'' sequences underlying 
them are discovered in this paper.
\begin{enumerate}
\item[1.] Introduction \\[-18pt]
\item[2.] Recurrence formula from geometry \\[-18pt]
\item[3.] More on lens geometry \\[-18pt]
\item[4.] Basic algebraic properties of lens sequences \\[-18pt]
%\item[5.] Binet-type formulas for lens sequences \\[-18pt]
%\item[6.] Fibonacci, and quadratic Pisot numbers \\[-18pt]
%\item[7.] Chebyshev polynomials \\[-18pt]
\item[5.] Underground sequences \\[-18pt]
\item[6.] Summary \\[-18pt]
\end{enumerate}
\end{abstract}

\theoremstyle{plain}
\newtheorem{theorem}{Theorem}[section]          %{th:
\newtheorem*{thma}{Theorem A}
\newtheorem{corollary}[theorem]{Corollary}      %{cor:

\theoremstyle{definition}
\newtheorem{example}[theorem]{Example}		%{exm:
\newtheorem{exercise}[theorem]{Exercise}	%{exr:
\newtheorem{definition}[theorem]{Definition}
\newtheorem{proposition}[theorem]{Proposition}
\newtheorem{remark}[theorem]{Remark}

\numberwithin{equation}{section}
\numberwithin{figure}{section}
\numberwithin{table}{section}

\newcommand{\m}{\hphantom{m}}
\newcommand{\N}{{\mathbb N}}
\newcommand{\Z}{{\mathbb Z}}
%

%%%%%%%%%%%%%%%%%%%%%%%%%%%%%%%%%%%%%%
% Section: introduction
%%%%%%%%%%%%%%%%%%%%%%%%%%%%%%%%%%%%%%

\section{Introduction} \label{S1}
We investigate a new family of integer sequences. They are generated by a 
geometric construction, which we now describe.

\begin{figure}[h] %================================
\[
\includegraphics[scale=1.2]{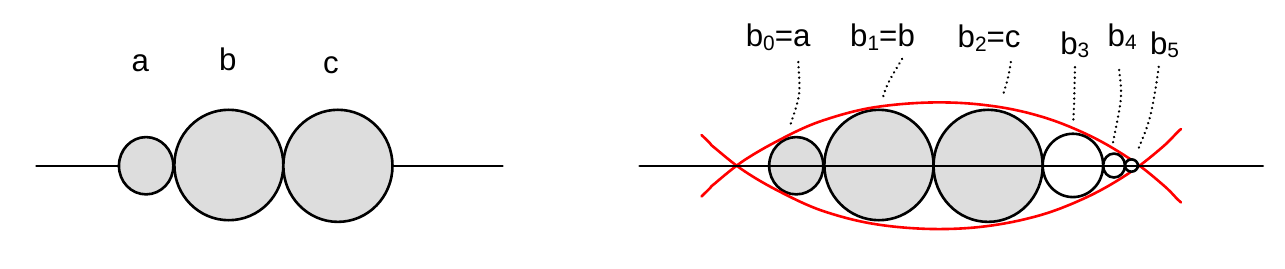} 
\]
\caption{Three circles determine a sequence}
\label{fig:1.1}
\end{figure}

Start with three circles of curvatures $a$, $b$, and $c$ centered on the same line, so 
that pairs of consecutive circles are tangent, as in the left side of Figure 
\ref{fig:1.1}. The three circles determine a pair of congruent circles that are 
simultaneously tangent to the original triple (right side of Figure \ref{fig:1.1}). 
The common region formed by this pair defines a symmetric \textbf{lens}. 
Now, continue to inscribe circles inside the lens, as shown in Fig.\ \ref{fig:1.1}. 
The resulting chain of circles defines a bilateral sequence of 
curvatures $(b_{i})$,  $i\in \Z$. 
Sequences obtained this way will be called \textbf{lens sequences}.

\paragraph{Terminology.} The resulting circles form a \textbf{lens circle chain}. 
We shall say that a triplet of circles $(a,b,c)$ \textbf{generates} the lens 
sequence, and we will call it a \textbf{seed} of the sequence. Notice that 
\textit{any} three consecutive terms of a lens sequence define a seed. 
The two circles that form the lens will be called \textbf{lens circles}.

\paragraph{Notation.} Typically, we denote circles and their curvatures by the 
same letter (circle $a$ has curvature $a,$ i.e., radius $1/a$).

\vskip.15in
Our opening result (proved in the next section) is this: 

\begin{thma} 	
Let $a, b$ and $c$ be the curvatures of the initial three 
circles generating a lens sequence, $b\ne $0. Then the sequence is determined 
by the following inhomogeneous three-term recurrence formula:
\begin{equation}    \label{eq:1.1}
  b_{n}=\alpha \,b_{n-1} - b_{n-2}+\beta\, ,
\end{equation}
where $\alpha $ and $\beta $ are constants determined by the original 
triple: 
\begin{equation}    \label{eq:1.2}
  \alpha = \frac{ab+bc+ca}{b^2}-1 
  \quad \hbox{and} \quad 
  \beta = \frac{b^2-ac}{b}\,.
\end{equation}
In particular, if $b_{0}=a$ and $b_{1}=b$ then $b_{2}=c$\,. 
\end{thma}

Constants $\alpha $ and $\beta $ are ``invariants'' of the sequence -- their values may be determined from {\textnormal any} three consecutive terms of the sequence.

If $a$, $b$, $c$ as well as $\alpha $ and $\beta $ are integers, then $(b_{n})$ is an 
integer sequence. Surprisingly, this family of sequences includes a wide 
range of known sequences \cite{Slo}. However, for some of these sequences,
properties that we develop in this paper seem to be new. 
For now, let us look at a few examples.

\begin{example}[Vesica Piscis]  \label{exm:1.1} 
Starting with $(a,b,c)$ = (3,1,3) we get the recurrence formula 
\[
  b_{n} = 14 b_{n-1} - b_{n-2} - 8\, ,
\]
which produces
\begin{center}
  \dots \textbf{3, 1, 3}, 33, 451, 6273, 87363, 1216801, 16947843, 236052993, \dots,
\end{center}
a sequence known as \seqnum{A011922}. Note that if one starts with values twice as large,
the whole sequence is doubled:
\begin{center}
\dots \textbf{6, 2, 6,} 66, 902, 12546, 174726, 2433602, 33895686, 472105986, \dots
\end{center}
and the recurrence formula becomes $b_{n} = 14 b_{n-1} - b_{n-2} - 16$ (the 
same $\alpha $ and twice $\beta )$.  (The lens circles traverse each others center, 
forming a well-known figure of {\em Vesica Piscis}, hence the name of the example.)
\end{example}

\begin{example}[Golden Vesica]  \label{exm:1.2}
Start with $(a,b,c) = (1, 2, 10)$. Equations \eqref{eq:1.2} give $\alpha =7$ and $\beta  = -3$;
hence the sequence is generated by 
\[
  b_{n} = 7 b_{n-1} - b_{n-2} - 3
\]
and is
\begin{center}
  \textbf{1, 2, 10}, 65, 442, 3026, 20737, 142130, 974170, \dots
\end{center}
for positive $n$. This sequence, listed as Sloan's \seqnum{A064170}, is known for its interesting properties. 
Its terms are products of pairs of non-consecutive Fibonacci numbers: 
$1\cdot 2$, $2\cdot 5$, $5 \cdot 13$, $13\cdot 34$, \dots, etc. 
They also coincide with the denominators in a system of Egyptian fraction for ratios of consecutive Fibonacci numbers: 
$1/2 = 1/\mathbf{2}$, $3/5 = 1/\mathbf{2} + 1/\mathbf{10}$, $8/13 = 1/\mathbf{2} + 1/\mathbf{10} + 1/\mathbf{65}$, etc.
(The geometry of the lens relates to the golden proportion, hence the proposed name.)
\end{example}

\begin{example}  \label{exm:1.3} 
%\textbf{``Even subsquares''.} 
Triplet $(-1, 3, 15)$ gives 
\begin{center}
  \dots 99, 63, 35, 15, 3, \textbf{$-$1}, \textbf{3, 15}, 35, 63, 99, 143, \dots 
\end{center}
from the recurrence
\[
  b_{n} = 2 b_{n-1} - b_{n-2} + 8\,.
\]
The sequence (3, 15, 35, \dots) is known as \seqnum{A000466} and is defined by $b_{n} = 4n^{2} - 1$. The occurrence of 
negative curvatures will be explained later. 
\end{example}

\begin{example}  \label{exm:1.4} 
A lens sequence does not necessarily need to be 
symmetric. For instance the triple (2,1,3) produces the following 
bilateral sequence:
\begin{center}
  \dots, 12972, 1311, 133, 14, \textbf{2, 1, 3}, 24, 232, 2291, 22673, 224434, \dots
\end{center}
\end{example}

\begin{example}[More Examples]   \label{exm:1.5}
The lens sequences possess an ample 
diversity. They include such basic examples as (i) the powers of 2 
(\seqnum{A000079}), and (ii) triangular numbers (\seqnum{A000217}). 
\begin{enumerate}
\item[(i)] \textbf{1, 2, 4}, 8, 16, 32, 64, 128, 256, \dots 
        \hbox to 1.2in{}
	$b_{n} = 5/2 \,b_{n-1} - b_{n-2}$ 

\item[(ii)] \textbf{1, 3, 6}, 10, 15, 21, 28, 36, 45, 55, 66, 78, 91, \dots  
	\hbox to .65in{}
	$b_{n} = 2 b_{n-1} - b_{n-2} + 1$
\end{enumerate}
\end{example}

A more extensive list with references to OEIS \cite{Slo} is provided 
in Tables \ref{tbl:4.1} through \ref{tbl:4.3}.

\begin{remark} 
Certain circle packings, known as Apollonian gaskets \cite{Man}, result 
in \textit{integral} curvatures for all of the circles (see e.g., \cite{LMW}). 
One such gasket, an \textit{Apollonian Window} (see \cite{K1,K2}), 
is presented in Fig. \ref{fig:1.2}. 
Interestingly, it contains an infinite number of 
lens sequences, from which the three shown in Fig. \ref{fig:1.2} are especially conspicuous.
They correspond to the Examples \ref{exm:1.1}, \ref{exm:1.2}, and \ref{exm:1.3}, 
given above (up to scaling). This observation was the author's initial motivation for this study.
\end{remark}

\begin{figure}[h] %================================
\[
\includegraphics[scale=1.2]{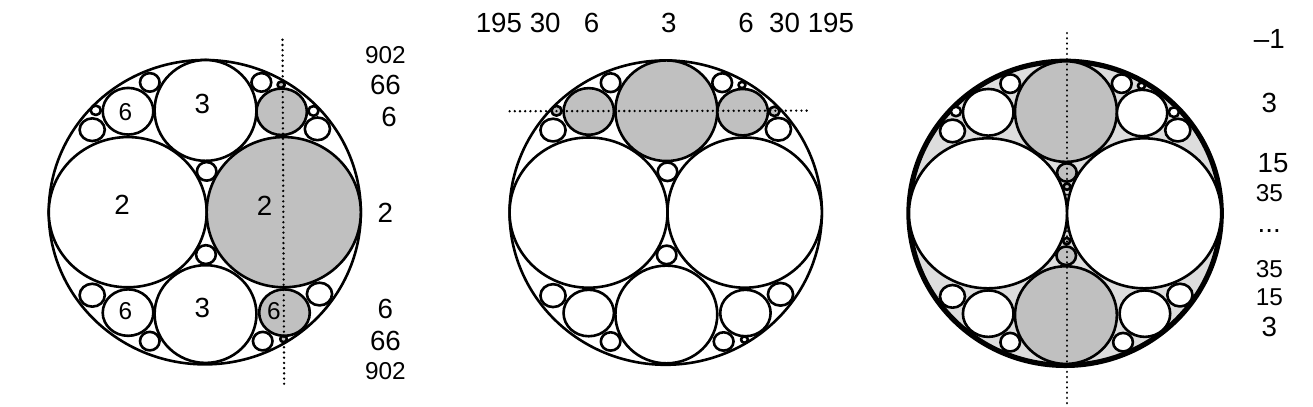} 
\]
\caption{Sequences \seqnum{A011922}, \seqnum{A064170}, 
and \seqnum{A000466} in the Apollonian window}
\label{fig:1.2}
\end{figure}
 
The negative term of Example \ref{exm:1.3} (the third one in Figure \ref{fig:1.2}) has 
a clear geometric meaning: it is the curvature of the greatest disc in the sequence, 
which --- unlike the other discs in the sequence --- contains the lens circles 
(as well as the rest of the sequence). 

\vskip.1in
A number of interesting features are common to all lens sequences: 

\paragraph{1. Limits.} In many cases, the limit of the ratios of  
consecutive entries is well defined. For instance, referring to the above 
examples:

\begin{tabular}{ll}
Example 1 [\seqnum{A011922}]: &$\displaystyle\lim_{n\to\infty } \dfrac{b_{n+1} }{b_n }
                                     =7+4\sqrt 3 = (2+\surd 3)^2$ \\[10pt]
Example 2 [\seqnum{A064170}]: &$\displaystyle\lim_{n\to\infty } \dfrac{b_{n+1} }{b_n }
                                     =\frac{7+3\sqrt 5 }{2} 
                                     =\left(\frac{1+\sqrt 5 }{2}\right)^4$\\[10pt]
Example 3 [\seqnum{A000466}]: &$\displaystyle\lim_{n\to\infty } \dfrac{b_{n+1} }{b_n } 
                                     = 1 $ \\[10pt]
\end{tabular}

\noindent
These numbers are examples of Pisot numbers 
and will be called characteristic constants of the sequences
denoted $\lambda=(\alpha+\sqrt{\alpha^2-4)}$.

\paragraph{2. Sums.} 
The reciprocals of curvatures are the circles' radii. 
Their sum is determined by the length of the lens. For instance:

\vskip.1in
\begin{tabular}{ll}
Example 1 [\seqnum{A011922}]: 
	&$\Sigma_{i=1}^\infty\, 1/b_{i} = 1 + 1/3 + 1/33 + 1/451 + \ldots 
		= \surd 3/2$  \\[8pt]
Example 2 [\seqnum{A064170}]: 
	&$\Sigma_{i=1}^\infty\, 1/b_{i} = 1 + 1/2 + 1/10 + 1/65 + \ldots 
		= (1+\surd 5)/2 = 1.618 \ldots$  \\[8pt]
Example 3 [\seqnum{A000466}]: 
	&$\Sigma_{n=1}^\infty\, 1/b_{n} = \sum _{n} \,1/(n^{2}-1) = 1/3 +1/15
		+ 1/35 + \ldots = 1$\\[8pt]
\end{tabular}

\paragraph{3. Binet-type formulas.} 
For $\alpha\not=2$, the curvatures may be expressed in terms a non-homogeneous 
Binet-type formulas:

\vskip.1in
\noindent
Example 1 [Vesica Piscis, \seqnum{A011922}]: \textbf{3 1 3} 33 451 6273 87363 {\ldots}, 
% $b_{n} = 14b_{n-1} - b_{n-2} - 16$
\[
  b_{n}  = \frac{4+(2+\sqrt{3})^{2n}+(2-\sqrt{3})^{2n}}{6}\, .
\]

\noindent
Example 2. [Golden Vesica [\seqnum{A064170}]: 
\ldots \textbf{2 1 2} 10 35 442 \ldots, 
\[
 b_{n} = \frac{3+ \left(\frac{1+\sqrt 5}{2}\right)^{4n}+\left(\frac{1-\sqrt 5}{2}\right)^{4n}}{5}
% \hbox{ where }	\lambda = \frac{7+3\sqrt 5 }{2}=\varphi^{4}\, ,
\]

\noindent
Example 3. Also non-symmetric lens sequences can be expressed this way. 
For instance, the sequence extended from $(6, 2, 3)$, which is 
(\ldots 2346 299 39 \textbf{6 2 3} 15 110 858 6747\ldots),
with the recurrence $b_{n} = 8b_{n-1} - b_{n-2} - 7$, may be obtained from 
\[
  b_{n} = \frac{(25-3\sqrt{15}) (4 + \surd 15)^n+(25+3\sqrt{15})
	(4 - \surd 15)^n}{60} + \frac{7}{6}
\]

\bigskip

The above properties are known for some of the sequences, but now they acquire a geometric interpretation. Other related concepts include geometry of inversions, Chebyshev polynomials, etc. 

The most remarkable and perhaps surprising property is that the integer lens sequences 
are ``shifted squares'' of yet deeper integer ``underground'' sequences. 
This discovery is the topic of the final section of this paper.

%-------------------------------------------------------------------------
\section{Recurrence formula from geometry} \label{S2}

In this section we prove Theorem A on the recurrence formula for lens 
sequences. A reader interested in the algebraic properties of these 
sequences may skip it without loss of continuity. 
\\

We shall need a theorem on circle configurations generalizing 
that of Descartes' theorem on ``kissing circles'' (\cite{Des,K1}). 
If C$_{1}$ and C$_{2}$ denote two circles of radii $r_{1}$ and $r_{2}$ respectively, and $d$ denotes 
the distance between their centers, then one defines a product of the circles as
\begin{equation}    \label{eq:2.1}
  \langle C_{1},C_{2}\rangle = \frac{d^2-r_1^2 -r_2^2 }{2r_1 r_2 }\,,
\end{equation}
which we propose to call the \textit{Pedoe product}. Its values for a few cases are shown in Figure \ref{fig:2.1}. 
For any four circles $C_{i}$, $i = 1,\ldots,4$, define a \textbf{configuration matrix} $f$ as the matrix with entries
\[
  f_{ij} = \langle C_{i},C_{j}\rangle\,,
\]
where the brackets denote the Pedoe inner product of circles.

\begin{figure}[h] %================================
\[
\includegraphics[scale=1.2]{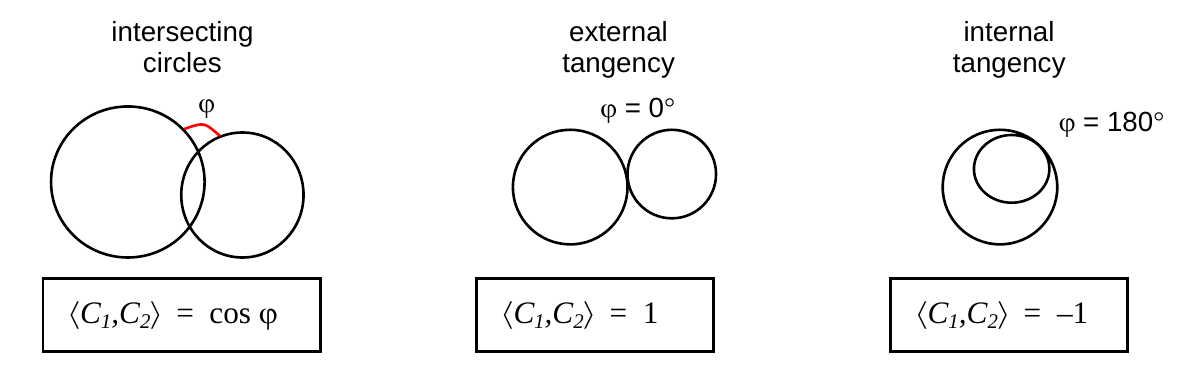} 
\]
\caption{Pedoe product of circles}
\label{fig:2.1}
\end{figure}

\begin{theorem}[\cite{K1}] \label{th:2.1} 
A configuration of four circles in general 
position satisfies the following quadratic equation
\begin{equation}    \label{eq:2.2}
  \textbf{\textit{b}}^{T}F \textbf{\textit{b}} = 0
\end{equation}
where $\textbf{\textit{b}} = [ b_{1}, b_{2}, b_{3}, b_{4} ]^{T}$ is the 
vector made of the curvatures of the four circles, and 
where $F= f^{-1}$ is the inverse of the configuration matrix.  
\end{theorem}

\begin{remark} 
Equation \eqref{eq:2.2} is only a fragment of the full matrix 
formula, which incorporates also the positions of the centers of the 
circles. For more on this theorem, its proof, and the associated Minkowski 
geometry of circles, see \cite{K1}. For our purposes Formula \eqref{eq:2.2} 
is sufficient.
\end{remark}

\paragraph{Notation.}
In the following by a \textbf{lens} we mean ``symmetric 
lens'' --- the intersection of the interiors (exteriors) of two congruent 
circles, called in this context \textbf{lens circles}. A \textbf{chain of 
circles} is a sequence of circles such that every two consecutive circles 
are tangent.

\vskip.2in
We are now ready to prove the basic result.

\begin{theorem}  \label{th:2.2} 
A sequence $(b_{n})$ of curvatures of a chain of circles inscribed in a lens satisfies a 
non-homogeneous linear recurrence formula of the form
\[
  b_{n+1} = \alpha  b_{n} - b_{n-1} +\beta 
\]
for some constants $\alpha $ and $\beta $, with 
\begin{equation}    \label{eq:2.3}
  \alpha = \frac{6-2K}{1+K} = \frac{8}{1+K}-2 
	\hbox{ and } 
  \beta =\frac{8A}{1+K} \, ,
\end{equation}
where $K$ is the Pedoe product of the two lens circles and $A=1/R$ is the 
curvature of each lens circle.
\end{theorem}

\begin{proof} 
Consider two consecutive circles in the lens, of curvatures 
say $a$ and $b$. Denote the curvatures of the circles that form the lens by $A$, 
and their Pedoe product by $K$ ($K=\cos \varphi $, if the circles intersect). 
The configuration matrix $f$ and its inverse are easy to find. In the case 
of converging lenses we can read it off from Fig. \ref{fig:2.2}a: 
\[
  f = \left[ {{\begin{array}{cc|cc}
 {-1} \hfill & K \hfill & {-1} \hfill & {-1} \hfill \\
 K \hfill & {-1} \hfill & {-1} \hfill & {-1} \hfill \\ \cline{1-4}
 {-1} \hfill & {-1} \hfill & {-1} \hfill & {+1} \hfill \\
 {-1} \hfill & {-1} \hfill & {+1} \hfill & {-1} \hfill \\
\end{array} }} \right]
\]
where the indices are ordered as $(A,A,x,y)$. Its inverse $F$ is easy to find, and the 
master equation \eqref{eq:2.1}, after multiplying by a factor of 8, becomes:
\[
\left[ {{\begin{array}{cccc}
 A \hfill \\
 A \hfill \\
 x \hfill \\
 y \hfill \\
\end{array} }} \right]^T
\left[ {{\begin{array}{cccc}
 {\tfrac{4}{K+1}}  & {\tfrac{-4}{K+1}}  & 2  & 2  \\[4pt]
 {\tfrac{-4}{K+1}}  & {\tfrac{4}{K+1}}  & 2  & 2  \\[4pt]
 2  & 2  & {K+1}  & {K-3}  \\[4pt]
 2  & 2  & {K-3}  & {K+1}  \\
\end{array} }} \right]
\left[ {{\begin{array}{*{20}c}
 A \hfill \\
 A \hfill \\
 x \hfill \\
 y \hfill \\
\end{array} }} \right]=\quad 0
\]
This quadratic equation is equivalent to:
\[
  (1+K)x^{2} + (1+K)y^{2} + 2(K-3) xy + 8Ax + 8Ay = 0\,.
\]
One may solve it for $y$ to get two solutions (corresponding to two signs at the square root):
\begin{equation}    \label{eq:2.4}
  y_{1,2} = \frac{4A+(K-3)x\pm 2\sqrt {2(1-K)x^2-8Ax+4A^2} }{1+K}
\end{equation}
Note that the two solutions $y_{1}$ and $y_{2}$ correspond to the two possible circles tangent to $x$: 
one on the left and one on the right. To eliminate radicals, add the two solutions:
\[
  y_{1} + y_{2} = \frac{6-2K}{1+K}x - \frac{8A}{1+K} \,.
\]
Since the triple $(y_{1}, x, y_{2})$ forms a sequence in a chain of inscribed circles, 
we may label these curvatures as $b_{n-1} = y_{1}$, $b_{n}=x$, and $b_{n+1} = y_{2}$, to get
\[
  b_{n+1} + b_{n-1} = \alpha\, b_{n}+\beta \,,
\]
which is equivalent to \eqref{eq:2.3}. The case of the diverging lens results from similar reasoning, 
with slightly different initial matrix $F$. 
\end{proof}

\begin{corollary} \label{cor:2.3} 
The sequence constants are related: $\alpha + R\,\beta  = -2$. 
\end{corollary}

Now let us see how three circles determine a sequence.

\begin{theorem} \label{th:2.4} 
Let $a, b$ and $c$ be curvatures of three consecutive circles inscribed in a lens. 
Then the sequence of the circle curvatures is determined by the 
following three-term recurrence formula:
\begin{equation}    \label{eq:2.5}
    b_{n} = \alpha \, b_{n-1} - b_{n-2} + \beta \, ,
\end{equation}
where $\alpha = \frac{ab+bc+ca}{b^2}-1$ and $\beta = 
\frac{b^2-ac}{b}$. If $b_{0}=a$ and $b_{1 }=b$ then $b_{2}=c$\,, 
that is, $\alpha b +\beta = a + c$.
\end{theorem}

\begin{proof} 
We apply Theorem \ref{th:2.1} in each of the three steps to a different quadruple of circles.

\begin{figure}[h] %================================
\[
\includegraphics[scale=1.2]{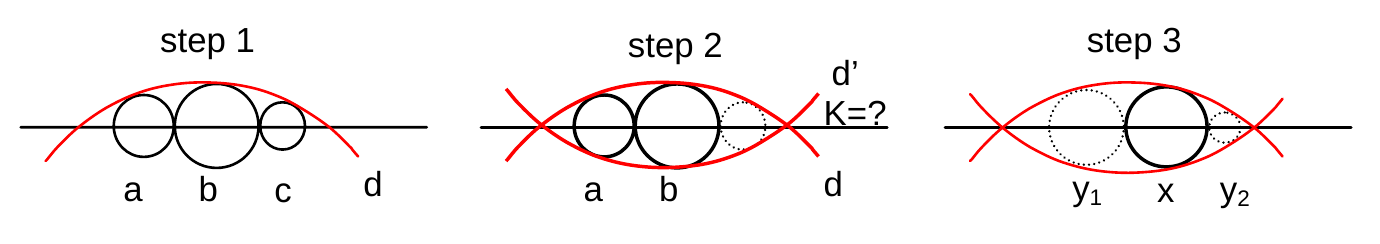} 
\]
\caption{The three steps of the proof of Thm.\ \ref{th:2.4}}.
\label{fig:2.2} 
\end{figure}

\paragraph{Step 1.} 
Consider a configuration of four circles: a triple of three 
consecutive circles in the chain, say $a$, $b$, $c$, plus one circle forming the lens, say $d=1/R$. 
The Pedoe product of two external circles $a$ and $c$ may be easily evaluated; 
since the distance between their centers is $(1/a+2/b+1/c)$, we have:
\[
  \langle a,c\rangle = \frac{\left( {\frac{1}{a}
	+ \frac{2}{b} + \frac{1}{c}} \right)^2
	- \left( {\frac{1}{a}} \right)^2 
	- \left( {\frac{1}{c}} \right)^2}{2\cdot \frac{1}{a}
	\cdot \frac{1}{c}} = 2\frac{ab+bc+ca}{b^2} + 1\,.
\]
Denote the main fraction of the last expression by $z = 
\frac{ab+bc+ca}{b^2}$\,. Then the configuration matrix and its inverse are
\[
  f = \left[ {{\begin{array}{ccc|c}
 {-1}  & \m1  & {2z+1}  & {-1}  \\
 \m1  & {-1}  & \m1  & {-1}  \\
 {2z+1}  & \m1  & {-1}  & {-1}  \\ \cline{1-4}
 {-1}  & {-1}  & {-1}  & {-1}  \\
\end{array} }} \right] , \quad
F = \tfrac{1}{4}\left[ {{\begin{array}{cccc}
 {-\tfrac{1}{z+1}}  & 1  & {\tfrac{1}{z+1}}  & {-1}  
\\
 \m1  & {-(z+1)}  & \m1  & {z-1}  \\
 {\tfrac{1}{z+1}}  & 1  & {-\tfrac{1}{z+1}}  & {-1}  
\\
 {-1}  & {z-1}  & {-1}  & {-(z+1)}  \\
\end{array} }} \right]
\]
(the order of entries is: \textit{abcd}).
Denoting $\mathbf{v} = [a,b,c,d]^{T}$ and solving the quadratic equation 
$\mathbf{v}^{T}F \mathbf{v} = 0$ for $d$ readily leads to
\[
  d=\frac{b(ac-b^2)}{ab+bc+ca+b^2}\,.
\]
This gives us the curvature of each of the two lens circles. 
Now we need to find the product of these two lens circles.

\paragraph{Step 2.} 
Use a quadruple of circles: $a$, $b$, and the two circles forming 
the lens, $d$ and $d'$. The latter two have the same curvature $d=d'$, 
the value of which we know from the previous step. The goal is to find 
the Pedoe product $K= \langle d, d'\rangle$\,. The configuration matrix 
and its inverse are
\[
  f=\left[ {{\begin{array}{cccc}
 {-1}  & \m1  & {-1}  & {-1}  \\
 \m1  & {-1}  & {-1}  & {-1}  \\
 {-1}  & {-1}  & {-1}  & K  \\
 {-1}  & {-1}  & K  & {-1}  \\
\end{array} }} \right] \quad 
F=\tfrac{1}{4}\left[ {{\begin{array}{cccc}
 {-1-K}  & {3-K}  & {-2}  & {-2}  \\[4pt]
 {3-K}  & {-1-K}  & {-2}  & {-2}  \\[4pt]
 {-2}  & {-2}  & {-\tfrac{4}{K+1}}  & {\tfrac{4}{K+1}} 
 \\[4pt]
 {-2}  & {-2}  & {\tfrac{4}{K+1}}  & {-\tfrac{4}{K+1}} 
 \\
\end{array} }} \right]
\]
(the order of indices agrees with $abdd'$). Applying vector $\mathbf{v} = [a,b,d,d]^{T}$ 
to the quadratic equation $\mathbf{v}^{T}F \mathbf{v}=0$ gives
\begin{equation}
\label{eq:2.6}
  K = \frac{8b^2}{(a+b)(b+c)}-1
\end{equation}

\paragraph{Step 3.} 
Now we can either build the matrix for configuration (c) in 
Figure \ref{fig:2.2} and mimic the proof of Theorem \ref{th:2.2}, 
or simply substitute for $K$ from \eqref{eq:2.6} in \eqref{eq:2.4} to get the result. 
\end{proof}

%-------------------------------------------------------------------------
\section{More on lens geometry}  \label{S3}

Although we are mainly interested in the algebraic properties of lens 
sequences, some geometric properties explicate their algebraic behavior. 
Below, we summarize basic facts.

\begin{proposition}  \label{prp:3.1} 
The \textbf{radius} $R$ of the lens circles is determined by three circles 
and may be expressed in terms of the sequence constants $\alpha $ and $\beta $:
\begin{subequations}  \label{eq:3.1}
\begin{equation}    \label{eq:3.1a}
  R = \frac{\alpha + 2}{-\beta} = \frac{(a+b)(b+c)}{(ac-b^2)b}\,.
\end{equation}
The Pedoe \textbf{inner product} of the lens circles is
\begin{equation}    \label{eq:3.1b}
  K = \frac{6 - \alpha}{2+\alpha} = \frac{8b^2}{(a+b)(b+c)} - 1 
    = \frac{1}{2} \left( {\frac{\delta}{R}} \right)^2 - 1 
    = \cos \varphi\,,
\end{equation}
where the last equation is valid if the circles intersect. 
The \textbf{length} $L$ of the lens, if defined, is 
\begin{equation}    \label{eq:3.1c}
  L = 2R\sqrt{\frac{\alpha-2}{\alpha+2}} 
    = -\frac{2\sqrt{\alpha^2-4}}{\beta} 
    = 2\frac{\sqrt {(a+b)(b+c)[(a+b)(b+c)-4b^2]}}{(ac-b^2)b}\,.
\end{equation}
The \textbf{separation} of the lens circles (distance between their centers) 
is
\begin{equation}    \label{eq:3.1e}
  \delta = \frac{4R}{\sqrt {\alpha + 2}}
	= -\frac{4\sqrt {\alpha +2} }{\beta } 
\quad\hbox{and}\quad
  \frac{\delta}{R} = \frac{4}{\sqrt {\alpha + 2}} \,.
\end{equation}
\end{subequations}
\end{proposition}

\begin{proof} 
All are direct corollaries of Theorem \ref{th:2.2} and simple 
geometric constructions. 
\end{proof} 

Figure \ref{fig:3.1} contains these findings for easy reference.

\begin{figure}[h!] %================================
\[
\includegraphics[scale=1.1]{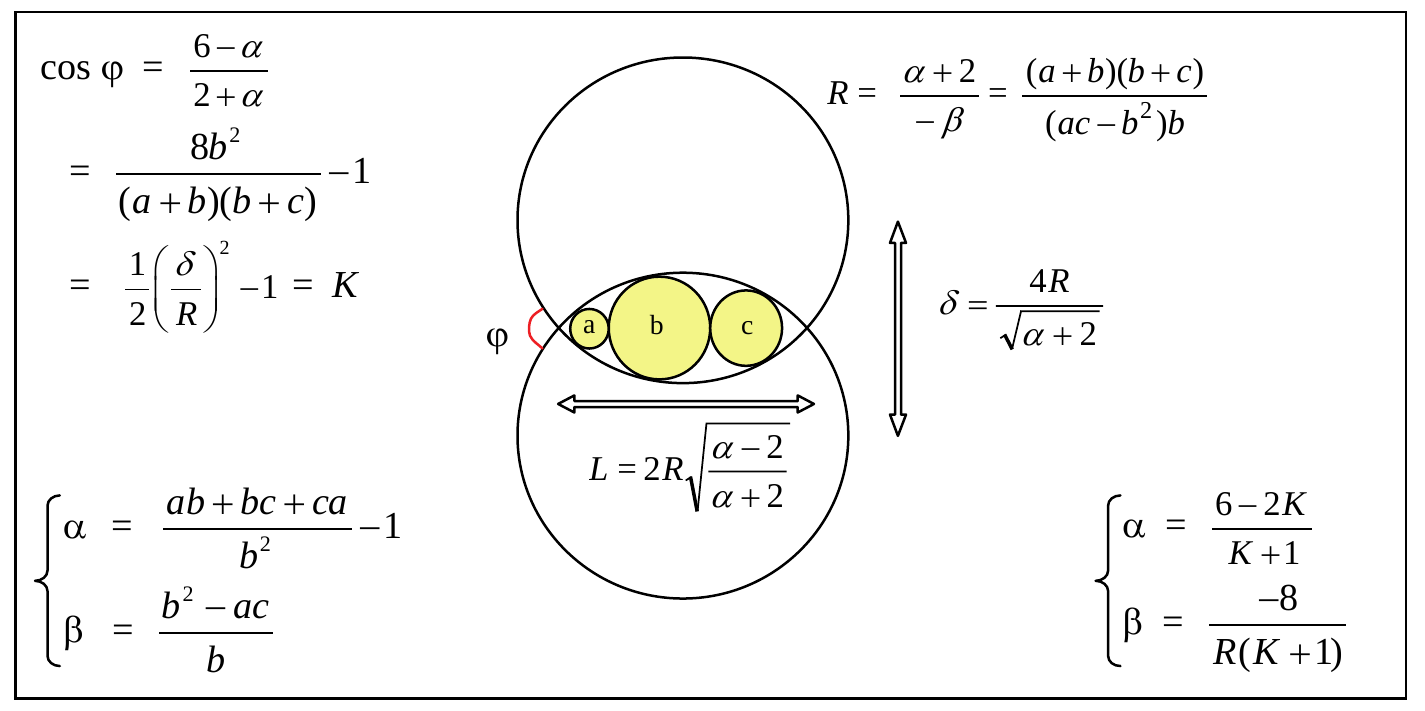} 
\]
\caption{Sequence constants and geometry of a lens}
\label{fig:3.1}
\end{figure}

\vskip.1in
Figure \ref{fig:3.2} categorizes a variety of geometric situations for a lens sequence. 
In the case of converging lenses, when two circles of radius $R$ intersect at angle $\varphi$, 
the recurrence formula is:
\[
  b_{n} = \left( \frac{8}{1 + \cos\varphi} - 2 \right) 
	b_{n-1} - b_{n-2} - \frac{1}{R} \frac{8}{1+\cos \varphi }\,.
\]
This answers the question of which lenses may lead to integer sequences. 
Indeed, denote $n=\frac{8}{1+\cos \varphi}$. Then $\alpha = n-2$, 
$\beta  = - n/R$. For $n$ to be an integer, $n\in \mathbb{N}$, we 
need $\cos \varphi = 8/n - 1$. 
Table \ref{tbl:3.1} shows some values.
 
\begin{figure}[h!] %================================
\[
\includegraphics[scale=1.2]{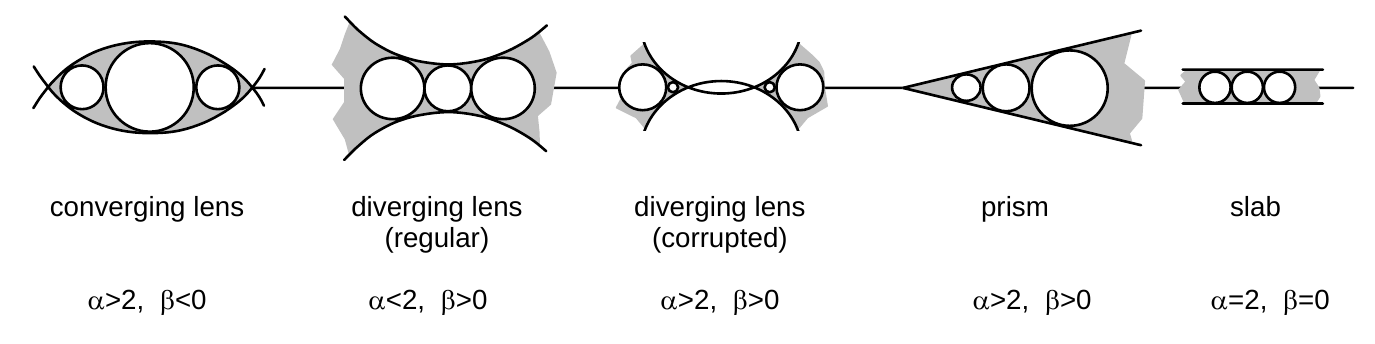} 
\]
\caption{Types of lenses and associated sequences}
\label{fig:3.2}
\end{figure}

\newcommand{\tblA}{%
\small{
\begin{tabular}{l|c|crlc|} \cline{2-6}
\hbox to .6in{} \quad &&&&&\\[-8pt]
&$n$  &$K$              &$\alpha$  	&        	&example \\[4pt] \cline{2-6}
\quad &&&&&\\[-8pt]
&0	&$\infty$	&$-2$	&periodic (order 2)	&$(1,-1)$ \\
&1	&7		&$-1$	&periodic (order 3)	&$(6,3,-2)$ \\ 
&2	&3		&0	&periodic (order 4)	&$(3,6,2,-1)$ \\ 
&3	&5/3		&1	&periodic (order 6)	&$(2,10,15,12,4,-1)$ \\[4pt] \cline{2-6}
\quad &&&&&\\[-8pt]
&4	&1		&2	&Example 1.3		&$(-1,3,15)$ \\ 
&5	&3/4		&3	&			&$(3,2,2,3)$ \\ 
&6	&1/3		&4	&       		&$(2,1,1,2)$ \\ 
&7	&1/7		&5	&			&$(5,2,2,5)$ \\ 
&8	&0		&6	&orthogonal		&$(3,1,1,3)$ \\ 
&9	&$-1/9$		&7	&golden Vesica		&$(2,1,2)$ \\ 
&10	&$-1/5$		&8	&			&$(4,1,1,4)$ \\ 
&11	&$-3/11$	&9	&			& \\ 
&12	&$-1/3$		&10	&			&$(3,1,2)$ \\ 
&13	&$-5/3$		&11	&			& \\ 
&14	&$-3/7$		&12	&			&$(6,1,1,6)$ \\ 
&15	&$-7/15$	&13	&			&$(4,1,2)$ \\ 
&16	&$-1/2$		&14	&Vesica Piscis		&$(3,1,3)$ \\ 
&17	&$-9/17$	&15	&			& \\ 
&18	&$-5/9$		&16	&			&$(8,1,1,8)$ \\ 
&20	&$-3/5$		&17	&			& \\[4pt] \cline{2-6}
\end{tabular}
}
}

\begin{table}[h!]  
\psset{xunit=.1in,yunit=.1in,runit=.05in}
\begin{pspicture}(0,-24)(0,22)
\rput(21,-2){\tblA}
\rput(48,12){$\left. \vbox to .4in{} \right\}$}
\rput(55,13){disjoint lens}
\rput(55,11){circles}
\rput(48,6){\Large$\boldsymbol\leftarrow$}
\rput(55,7){tangent lens}
\rput(55,5){circles}
\rput(48,-9){$\left. \vbox to 1.4in{} \right\}$}
\rput(56,-8){intersecting lens}
\rput(56,-10){circles}
\end{pspicture}
\caption{Admissible values of $\alpha$}
\label{tbl:3.1}
\end{table}

Figure \ref{fig:3.2} relates the geometry of lenses to the values 
of the sequence constant $\alpha $ and the Pedoe product $K$. 
Inspect Tables \ref{tbl:4.1} through \ref{tbl:4.3} (and associated figures) 
in the next section for various examples of lens sequences. 

Note that if $\alpha \leqslant 2$, then only external sequences are possible 
(corresponding to diverging lenses). Moreover, if $\alpha < 2$, 
then the integer sequence must be periodic. If $\alpha > 2$, then we can 
have two families of sequences: inner (inside a converging lens) or outer 
(outside the lens circles, i.e., inside a corrupted diverging lens 
(\textit{corrupted}, because of the missing central part)). 
In the case of the outer sequence we will have exactly one negative entry 
(the most external circle) or two adjacent ``0'' entries (two vertical lines).

%-------------------------------------------------------------------------
\section{Basic algebraic properties of lens sequences}  \label{S4}
Theorem A suggests the following definition:
 
\begin{definition} 	
A {\bf formal sequence} extended from a triplet $(a, b, c)$,
called a {\bf seed},  is defined 
by the following inhomogeneous three-term recurrence formula:
\begin{equation}    \label{eq:1.1}
  b_{n}=\alpha \,b_{n-1} - b_{n-2}+\beta\, ,
\end{equation}
where $\alpha $ and $\beta $ are constants determined by the original 
triple: 
\begin{equation}    \label{eq:1.2}
  \alpha = \frac{ab+bc+ca}{b^2}-1 
  \quad \hbox{and} \quad 
  \beta = \frac{b^2-ac}{b}\,.
\end{equation}
and $b_{0}=a$ and $b_{1}=b$.  (It follows that $b_{2}=c$)\,. 
\end{definition}
The values of the constants $\alpha$
and $\beta$ do not depend on the particular choice of the triplet of 
consecutive terms (seed).
Moreover, if $\alpha>-2$, then the sequence may be interpreted 
in terms of a chain of circles inscribed in a lens made by two disks each of 
curvature $R^{-1}=-\beta/(\alpha+2)$ separated by distance 
$\delta=4R/\sqrt{\alpha+2}$.   
Formal integer lens sequences exist also for $\alpha<-2$
(see Table 4.4 for examples), but in such a case the geometric interpretation 
is unclear as the distance between the lens circles becomes imaginary. 

In general, lens sequences take real values.  
However, if any three terms of a lens sequence are rational, 
so is the whole sequence. 
The question whether a particular seed produces an integer sequence will 
be answered for now this way:

\begin{proposition}[Integrality Criterion 1] \label{prp:ic} 
If $b\vert ac$ and $b^{2}\vert(ab+bc+ca)$, for any $a,b,c\in \mathbb{N}$, 
then a lens sequence extended from $(a,b,c)$ consists of integers.
\end{proposition}

Here are basic properties of lens sequences:

\begin{proposition}  \label{prp:4.1} 
Let $(b_{n})$ be a lens sequence. Then the following holds:
\begin{enumerate}
\item[(i)] Sequence $(b_{n})$ satisfies a homogeneous 4-term linear recurrence 
formula
\begin{equation}    \label{eq:4.1}
  b_{n} = (\alpha +1) b_{n-1} - (\alpha +1) b_{n-2}+b_{n-3}\,.
\end{equation}
\item[(ii)] If $\alpha \ge 2$ then the sum of the reciprocals converges and equals:
\begin{equation}    \label{eq:4.2}
  \sum\limits_{n=-\infty}^{\infty} {\frac{1}{b_n }} = \frac{2\sqrt {\alpha ^2-4} }{-\beta }
                                                    = \frac{L}{2}\,.
\end{equation}
\item[(iii)] If $\alpha >$2 then the limit of the ratios of consecutive 
terms exists and equals:
\begin{equation}    \label{eq:4.3}
  \lambda = \frac{\alpha +\sqrt {\alpha ^2-4}}{2}
          = \frac{\sqrt {\alpha +2}+\sqrt{\alpha-2}}{2} \,.
\end{equation}
\item[(iv)] If $\alpha\not= 2$, then the lens sequence generated from 
a seed $(a,b,c)$ has the following Binet-like formula 
\begin{equation}    \label{eq:5.1}
  b_{n}=w\lambda^n+\bar {w}\bar {\lambda }^n+\gamma
\end{equation}
where 
\[
  \lambda = \frac{\alpha +\sqrt {\alpha ^2-4} }{2}
	\qquad
	\bar {\lambda } = \frac{\alpha -\sqrt {\alpha ^2-4} }{2}.
\]
and where
\[
  w = \frac{a-2b+c}{2(\alpha -2)}+\frac{c-a}{2(\alpha ^2-4)}
	\sqrt {\alpha ^2-4} , 
	\qquad 
	\gamma = \frac{-\beta }{\alpha -2}
\]
and $\bar {w}$ and $\bar {\lambda }$ denote conjugates of $w$ and $\lambda $ in 
$\mathbb{Q(}\sqrt {\alpha ^2-4} )$, respectively. 
In particular, $(a,b,c)=(b_{-1},b_0,b_1)$.
\end{enumerate}
\end{proposition}

\begin{proof} 
(i) Elementary. 
(ii) From the geometry of lenses, cf.\ \eqref{eq:3.1c}. 
See also Figure \ref{fig:4.1}. 
(iii)~Divide the recurrence formula by 
$b_{n-1}$ to get
\[
  b_{n}/b_{n-1} = \alpha - b_{n-2}/b_{n-1}+\beta/b_{n-1}\, .
\]
For large values of $n$, since the sequence is divergent, the last term becomes 
irrelevant and the equation becomes $\lambda =\alpha - 1/\lambda $, 
or simply 
\begin{equation}    \label{eq:4.4}
  \lambda ^{2} - \alpha \lambda + 1 = 0\,,
\end{equation}
with the solution as above. Figure \ref{fig:4.1} provides the geometric insight, 
which also relates $\lambda $ to the lens angle via similar triangles. 
(iv)~Define a new sequence whose entries are shifted by a constant, namely
$a_{n} = b_{n}+\beta/(\alpha -2)$.
The sequence $(a_{n})$ satisfies a homogeneous three-term recurrence formula
$a_{n}=\alpha \, a_{n-1} - a_{n-2}$,
which resolves to \ref{eq:5.1} by the standard procedure. 
\end{proof}

%\[
%  \left[ {{\begin{array}{c}
%	a_{n+1} \\
%	a_n     \\
%	\end{array} }} \right] = 
%  \left[ {{\begin{array}{cr}
%	\alpha & {-1} \\
%	1  & 0  \\
%	\end{array} }} \right]
%	\quad
%	\left[ \begin{array}{c}
%	a_n     \\
%	a_{n-1} \\
%	\end{array}  \right]
%\]

The value of $\lambda $ (given by \eqref{eq:4.3}) will be called the 
\textbf{characteristic constant} of the sequence. The ring over 
rational numbers generated by $\sqrt {\alpha ^2-4} $ plays an 
important role in other properties of lens sequences, 
as we shall soon see. Note that the sequence constant $\alpha $ may 
be expressed in terms of the characteristic constant in a graceful way:
\[
  \alpha =  \lambda +\frac{1}{\lambda}
\]
\begin{figure}[t] %================================
\[
\includegraphics[scale=.9]{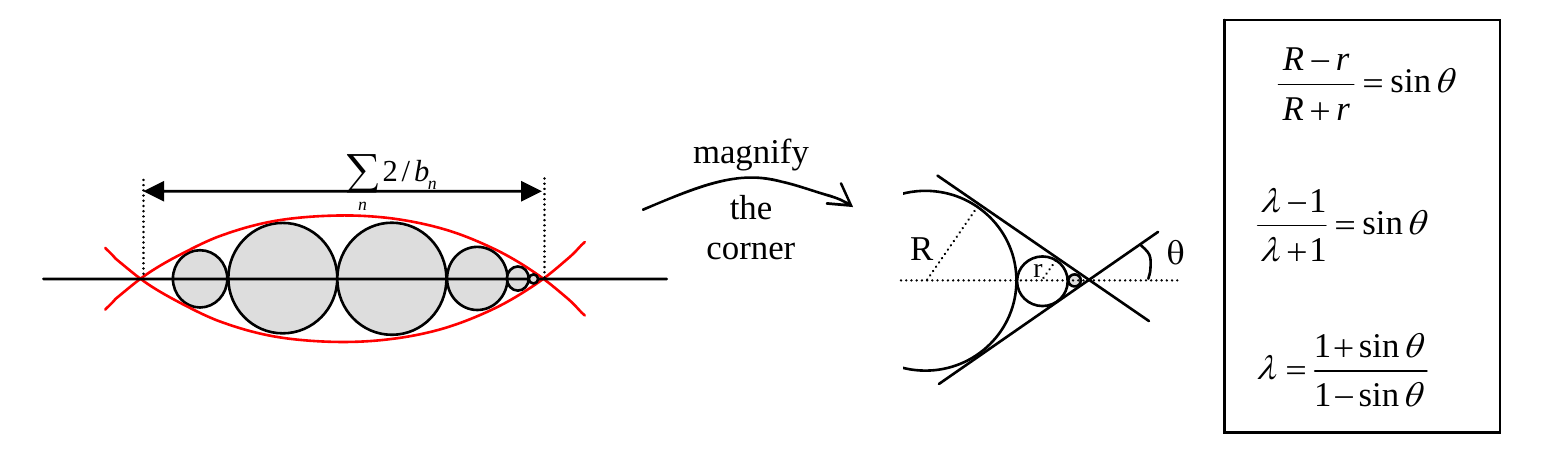} 
\]
\caption{Geometric meaning of $\sum\limits_n {1/b_n }$ and 
$\lambda $ (see Proposition \ref{prp:4.1})}
\label{fig:4.1}
\end{figure}

%------------------------------------------------------------
\subsection*{Alternative generating formulae}

%%%%%%%%%%%%%%%%%%%%%%%%%%%%%%%%%%%%%%%%%%%%%%%%%%%%%%%%%%%%%%%%%%%
                                                                  %
\def\kill{
\begin{corollary}  	\label{cor:4.2} 
A lens sequence may be determined from each of the following nonlinear three-step formulas:
\begin{alignat}{2}
  b_{n}&=\frac{(\alpha +1)b_{n-1}^2 -b_{n-1} b_{n-2} }{b_{n-1} +b_{n-2} } 
	&\qquad &\hbox{(only $\alpha$ involved)}  
			\label{eq:4.5}			\\
  b_{n}&=\frac{b_{n-1} (b_{n-1} -\beta )}{b_{n-2} } 
	&\qquad &\hbox{(only $\beta$ involved)}
			\label{eq:4.6}
\end{alignat}
\end{corollary}

\begin{example} 	\label{exm:4.3}
Sequence \seqnum{A064170} of Example \ref{exm:1.2} with 
entries 1, 2, 10, 65, 442, 3026, \dots\ etc. (\textit{Golden Vesica}) 
may be obtained from either of these two formulas: 
\[
  b_{n}=\frac{b_{n-1} (b_{n-1} +3)}{b_{n-2} }, \qquad 
  b_{n} = \frac{8b_{n-1}^2 -b_{n-1} b_{n-2} }{b_{n-1} +b_{n-2} }\,.
\]
\end{example}
}
                                                                                  %     
%%%%%%%%%%%%%%%%%%%%%%%%%%%%%%%%%%%%%%%%%%%%%%%%%%%%%%%%%%%%%%%%%%%%%%%%%%%%%%%%%%%

Note that the three term formula \eqref{eq:1.1}, with given coefficients $\alpha $ 
and $\beta $, requires only two initial entries to produce a sequence. 
Yet not \textit{all} such initial values will produce a lens sequence of the 
type under discussion. This is because arbitrary initial values $b_{0}$ and $b_{1}$ do not 
need to be geometrically inscribable into a lens defined by $\alpha $ and $\beta$ 
as two consecutive circles. The following will clarify the situation:

\begin{proposition}[Compatibility condition]  	\label{prp:4.4}  
Two consecutive circles $a$ and $b$ in a lens chain satisfy the following condition: 
\begin{equation}    \label{eq:4.7}
  a^{2}+b^{2}=\alpha ab + \beta (a+b)\, .
\end{equation}
\end{proposition}

\begin{proof} 
Eliminate $c$ from the expressions for $\alpha$ and $\beta$ in (\ref{eq:1.2}), and simplify. 
\end{proof}

The above formula may actually be used as an alternative definition of lens sequences.
Indeed,

\begin{proposition}  	\label{prp:def} 
Consider the following properties:
\begin{enumerate}
\item[(a)] Recurrence $b_{n+1} = \alpha b_{n} - b_{n-1} +\beta$\,, for all $n,$  \\[-20pt]
\item[(b)] Constants $\alpha = \frac{b_{n-1}b_{n}+b_{n}b_{n+1}+b_{n+1}b_{n-1}}{b_{n}0^2}-1 
                      \quad \hbox{and} \quad 
                     \beta = \frac{b_n^2-b_{n-1}b_{n+1}}{b_n}$. \\[-20pt]
\item[(c)] $a^{2} + b^{2} = \alpha\, ab +\beta (a+b)$\,, for $a=b_{n}$ and $b=b_{n+1}$.
\end{enumerate}
The following descriptions of a sequence $(b_i)$ are equivalent:
\begin{enumerate}
\item[(i)] Recurrence (a), and constants (b) for some $n$ 
           \quad (definition of a lens sequence); \\[-20pt]
\item[(ii)] Recurrence (a) and compatibility condition (c) for some $n$;  \\[-20pt]
\item[(iii)] Any of the two constant formulas (b) for all $n$; \\[-20pt]
\item[(iv)] Compatibility condition (c)  for all $n$. 
\end{enumerate}
\end{proposition}

\begin{proof} 
Let us show that the  
compatibility condition together with the 
recurrence theorem imply our standard formulas for both $\alpha$ and $\beta$.
Starting with (ii) we get
\[
\begin{aligned}
  b_{n+1}^2+b_n^2 &= \alpha \,b_n b_{n+1} +\beta (b_n +b_{n+1} ) \\
   &= b_{n+1} (\alpha \,b_n +\beta )+\beta b_n   \\
   &= b_{n+1} (b_{n+1} +b_{n-1} )+\beta b_n \,.
\end{aligned}
\]
Subtracting $b_{n+1}^{2}$ from both sides, we get
$
  b_n ^2=b_{n+1} b_{n-1} +\beta\, b_n \, ,
$
which gives
\[
  \beta = \frac{b_n^2 -b_{n-1} b_{n+1} }{b_n }\,.
\]
To get $\alpha $, substitute this result in the compatibility condition and 
simplify.  The other equivalences follow easily. 
\end{proof}

%%%%%%%%%%%%%%%%%%%%%%%%%%%%%%%%%%%%%%%%%%%%%%%%%%%%%%%%%%%%
                                                           %
\def\killll{ 
\begin{corollary} 	\label{cor:4.7} 
If $b_{n}$ is an entry in the lens sequence with constants $\alpha$ and $\beta$, 
then the two adjacent entries are
\begin{equation}    \label{eq:4.8}
  b_{n\pm 1} = \frac{b_n \alpha +\beta \pm \sqrt {(\alpha ^2-4)\;
	b_n^2+2(\alpha +2)\beta \;b_n +\beta ^2} }{2}
\end{equation}
{\normalfont(}two solutions are associated with the two possible signs, corresponding 
to the two directions for the sequence{\normalfont)}. {\normalfont[}Compare with the 
result \eqref{eq:2.4}{\normalfont]}. 
\end{corollary}
}
                                                        %
%%%%%%%%%%%%%%%%%%%%%%%%%%%%%%%%%%%%%%%%%%%%%%%%%%%%%%%%%

Here is yet another intriguing formula that will prove itself handy later.

\begin{proposition}  	\label{prp:4.8} 
Constant $\alpha $ has an alternative form involving any four consecutive entries 
of a lens sequence:
\begin{equation}    \label{eq:4.9}
  \alpha = \frac{b_{n-1} }{b_n }+\frac{b_{n+2} }{b_{n+1} }\,.
\end{equation}
\end{proposition}

\begin{proof} 
Start with the formula for $\beta $ and express it as follows:
\[
  \beta = \frac{b_n^2 -b_{n+1} b_{n-1} }{b_n } 
	= b_n -\frac{b_{n+1} b_{n-1} }{b_n }  
	\quad\Rightarrow\quad
   b_{n} - \beta  = \frac{b_{n+1} b_{n-1} }{b_n }\,.
\]
Use the recurrence formula to modify the left-hand side of the last 
equation, 
\[
  \alpha\, b_{n+1} -b_{n+2} = \frac{b_{n+1} b_{n-1} }{b_n }\,.
\]
Now extract $\alpha $ to get \ref{eq:4.9}.
\end{proof}
%%%%%%%%%%%%%%%%%%%%%%%%%%%%%%%%%%%%%%%%%%
                                         %
\def\killl{
\begin{corollary}  	\label{cor:4.9} 
Lens sequences satisfy the four-step recurrence formula
\[
  b_{n+2 }= b_{n+1} (\alpha\, b_{n}-b_{n-1})/b_{n}\, .
\]
\end{corollary}
}
                                     %

%------------------------------------------------------------
\subsection*{Integrality condition}

Lens sequences are self-generating in the sense that any three consecutive entries $(a,b,c)$, 
a seed,  determine the whole sequence (unless $b=0$). 
%It is only natural to call such a triple a \textbf{seed} of the sequence. 
The question is how to choose seeds $(a,b,c)$ in 
order to obtain  lens sequences that are \textit{integer}.
Below we give only a partial answer to this problem of integrality conditions; 
the last section will provide the solution.

\begin{definition} 	\label{def:4.10} 
An integer lens sequence is \textit{primitive} if the common divisor of three consecutive 
entries is 1. 
\end{definition}

\begin{proposition}  \label{prp:4.11} 
If $\gcd(b_{k}, b_{k+1}, b_{k+2})=n$ holds for some $k$, then it holds for all $k \in \Z$. 
\end{proposition}

\begin{proof} If $n$ divides each of $(b_{k}, b_{k+1}, b_{k+2})$, then it divides $\beta $ 
of the recurrence formula. Hence it divides the neighboring terms $b_{k+2}$ and $b_{k-1}$. 
By induction, $n$ divides every term of the sequence. 
\end{proof}

Recall that to insure that $\alpha $ and $\beta $ are integers, we need to 
choose $(a,b,c)$ so that $b\vert ac$ and $b^{2}\vert ab+bc+ca$, or,  equivalently, 
that $b^{2}\vert (a+b)(b+c)$. Thus triples of the form $(a, 1, c)$ always 
generate integer sequences for any $a,c\in \mathbb{N}$\,. 
For now, let us review the following families of integer lens sequences:

\paragraph{A.} A sequence is called \textbf{central} if it contains a triple of the 
form $(a,b,a)$. Only $b=\pm 1$ leads to primitive integer sequences.

\paragraph{B.} A lens sequence is called \textbf{bicentral}, if it contains a 
quadruplet of the form $(a,b,b,a)$. Only if $b$ is chosen from $\{0, 1, 2\}$, does 
a primitive integer sequence result. (For the case $b=0$, the seed needs to be chosen in 
the form $(0, a, c)$). 

\vskip.1in
In either case A or case B, the sequence is called \textbf{symmetric}. 
Table \ref{tbl:4.1} shows examples of symmetric lenses for small values of the 
initial terms. (Only the right tail is displayed.)

\paragraph{C.} Here is a method of getting a not-necessarily symmetric integer sequence: 
choose arbitrarily a couple $(a,b)$ and some integer $k$ (only one of $a$ or $b$  
can be negative). Then a seed $(a,b,c)$ with
\begin{equation}    \label{eq:4.10}
  c = b(bk-1)\,,
\end{equation}
will generate an integer sequence. Indeed, calculate the recurrence 
constants $\alpha $ and $\beta $ from \eqref{eq:4.10} to get the integer values
\begin{equation}    \label{eq:4.11}
\begin{aligned}
  \alpha &= (a+b)k - 2 	\\
  \beta &= (a+b) - ab \, k\,.
\end{aligned}
\end{equation}
The triple of integers $[a, b; k]$ will 
be called the \textbf{label} of a lens sequence of this type.
(Note that it includes the symmetric lens sequences as a special case;
for examples of non-symmetric sequences, see Table \ref{tbl:4.2}.)  \\

The pair $(a,b)$ in the label may be chosen so that it contains the smallest element 
of the sequence. \\

\clearpage

%%%%%%%%%%%%%%%%%%%%%%%%%%%%%%%%  TABLES AND FIGURES  %%%%%%%%%%%%%%%%%%%%%%%%%
\begin{table}  	%tbl:4.1
\footnotesize  %\scriptsize %
\begin{tabular}{lllclll} \hline
\quad &&&&& \\[-9pt]
   &central \\[-3pt]
   &elements &constants & OEIS nr &sequence & label and symbol \\[3pt]  
\hline
\ &&&&& \\[-9pt]
\multicolumn{3}{l}{Central sequences} \\[3pt]
1. &$(2, 1, 2)$ &$\alpha =7, \beta =-3$      &\seqnum{A064170} 
   &{\scriptsize $1, 2, 10, 65, 442, 3026, 20737, 142130, {\dots}$}           
   &[1,2;3]     & ${}^3(1,1)^3$  \\
2. &$(3, 1, 3)$ &$\alpha =14, \beta =-8$     &\seqnum{A011922} 
   &{\scriptsize $1, 3, 33, 451, 6273, 87363, 1216801, 16947843, {\ldots}$}   
   &[1,3;4]     &${}^4(1,1)^4$ \\
3. &$(4, 1, 4)$ &$\alpha =23, \beta =-15$    &---  
   &{\scriptsize $1, 4, 76, 1729, 39676, 910804, 20908801, {\ldots}$} 
   &[1,4;5]     &${}^5(1,1)^5$ \\
4. &(5, 1, 5) &$\alpha =34, \beta =-24$    &---  
   &{\scriptsize $1, 5, 145, 4901, 166465, 5654885, 192099601, {\ldots}$}   
   &[1,5;6]       &${}^6(1,1)^6$ \\
5. &$(3,-1,3)$ &$\alpha =2, \beta =8$    &\seqnum{A000466} 
   &{\scriptsize $-1, 3, 15, 35, 63, 99, 143, 195, 255, 323, 399, {\ldots}$}   
   &$[-1,3;2]$ &${}^2(1,3)^2$  \\
6. &$(4,-1,4)$ &$\alpha =7, \beta =15$   &\seqnum{A081078} 
   &{\scriptsize $-1, 4, 44, 319, 2204, 15124, 103679, 710644, {\ldots}$} 
   &$[-1,4;3]$      &${}^3(1,4)^3$ \\
7. &$(5, -1, 5)$ &$\alpha =14, \beta =24$  &--- 
   &{\scriptsize $-1, 5, 95, 1349, 18815, 262085, 3650399$}  
   &$[-1,5;4]$     &${}^4(1,5)^4$ \\  %50843525
8. &$(6, -1, 6)$ &$\alpha =23, \beta =35$  &---  
   &{\scriptsize $-1, 6, 174, 4031, 92574, 2125206, 48787199, {\ldots}$}   
   &$[-1,6;5]$     &${}^5(1,6)^5$ \\
\quad &&&&& \\[-9pt]
\multicolumn{3}{l}{Bicentral sequences} \\[3pt]
9. &$(2, 1, 1, 2)$ &$\alpha =4,\beta = -1$  &\seqnum{A101265}   
   &{\scriptsize $1, 2, 6, 21, 77, 286, 1066, 3977, 14841, {\ldots}$}   
   &[1,1;3]           &${}^2(1,1)^3$\\
10. &$(3, 1, 1, 3)$ &$\alpha =6, \beta = -2$  &\seqnum{A011900}  
    &{\scriptsize $1, 3, 15, 85, 493, 2871, 16731, 97513, 568345, {\ldots}$}  
    &[1,1;4]     &${}^2(1,1)^4$\\
11. &$(4, 1, 1, 4)$ &$\alpha =8, \beta = -3$  &---  
    &{\scriptsize $1, 4, 28, 217, 1705, 13420, 105652, 831793, {\ldots}$}   
    &[1,1;5]       &${}^2(1,1)^5$\\
12. &$(5,1, 1, 5)$ &$\alpha =10, \beta = -4$  &\seqnum{A054318}  
    &{\scriptsize $1, 5, 45, 441, 4361, 43165, 427285, 4229681, {\ldots}$} 
    &[1,1;6]        &${}^2(1,1)^6$\\
13. &$(3,2, 2, 3)$ &$\alpha =3, \beta = -1$  &\seqnum{A032908}  
    &{\scriptsize $2, 3, 6, 14, 35, 90, 234, 611, 1598, 4182, 10947, {\ldots}$} 
    &[2,3;1]   &${}^5(1,2)^1$\\
14. &$(5, 2, 2, 5)$ &$\alpha $=5, $\beta $= --3  &---  
    &{\scriptsize $2, 5, 20, 92, 437, 2090, 10010, 47957, 229772, {\ldots}$} 
    &[2,5;1]      &${}^7(1,2)^1$\\
15. &$(0, 0, 1,3)$ &$\alpha =2, \beta = 1$  &\seqnum{A000217}  
    &{\scriptsize $0, 1, 3, 6, 10, 15, 21, 28, 36, 45, 55, 66, {\ldots}$} 
    &[0,1;4]         &${}^1(1,1)^4$\\
16. &$(0, 0, 1,4)$ &$\alpha =3, \beta = 1$  &\seqnum{A027941}  
    &{\scriptsize $0, 1, 4, 12, 33, 88, 232, 609, 1596, 4180, {\ldots}$}
    &[0,1;5]          &${}^1(1,1)^5$\\
17. &$(0, 0, 1,5)$ &$\alpha =4, \beta = 1$  &\seqnum{A061278}  
    &{\scriptsize $0, 1, 5, 20, 76, 285, 1065, 3976, 14840, 55385, {\ldots}$} 
    &[0,1;6]    &${}^1(1,1)^6$\\[3pt] 
\hline
\end{tabular}
\caption{Examples of symmetric lens sequences}
\label{tbl:4.1}
\end{table}

\begin{table}  	%tbl:4.2
\begin{center}
\footnotesize
\begin{tabular}{rlllll}
  \hline
  \quad &&&& \\[-9pt]
  &seed &constants  &sequence & label and symbol \\[2pt] 
  \hline
\quad &&&& \\[-9pt]
1a. &$(3,1,2)$ &$\alpha =10,\, \beta = -5$  
    &24, 3, 1, 2, 14, 133, 1311, 12972, 128404, 1271063,...     &[1,2;4]   &${}^3(1,2)^4$ \\
b.  &$(2,1,3)$ &      
    &14, 2, 1, 3, 24, 232, 2291, 22673, 224434, 2221662,...     &[1,3;3]   &${}^4(1,3)^3$\\[3pt]

2. &$(5,3,6)$ &$\alpha =6,\, \beta = -7$  
    &\ldots, 108, 20,  5, 3, 6, 26, 143, 825, 4800, 27968, ...  &[5,3;1]  &${}^8(1,3)^1$ \\[3pt]
3.  &$(3, -1, 4)$ &$\alpha =4,\, \beta = 11$  
    &\ldots, 403, 104, 24, 3, $-$1, 4, 28, 119, 459, 1728,{\ldots}       &[$-$1,4;2] &${}^3(1,4)^2$\\ [3pt] %6464,
%    &   &    &$-$1, 3, 24, 104, 403, 1519, 5684, 21228, 79239, {\ldots} &[$-$1,3;3] &${}^2(1,3)^3$\\
%
4.  &$(15, 12, 20)$ &$\alpha =4,\, \beta = -13$  
    &\ldots, 400, 112, 35, 15, 12, 20, 55, 187, 680, 2520, ...           &--- &${}^2(4,5)^3$ \\[3pt]
5.  &$(21, 6, 10)$ &$\alpha =10,\, \beta = -29$  
    &\ldots, 16796, 1700, 175, 21, 6, 10, 65, 611, 6016, \ldots          &--- &${}^4(2,5)^3$ \\[3pt]

6. &$(1,2,4)$ &$\alpha = 5/2,\, \beta = 0$  
   &\ldots, 1, 2, 4, 8, 16, 32, 64, 128, 256, 512, 1024, {\ldots}   &[1,2;3/2]  &${}^3(1,2)^{2/3}$\\[3pt] \hline
\end{tabular}
\caption{Examples of non-symmetric sequences. Example 6 (\seqnum{A000079}) is integer in one direction only.}
\label{tbl:4.2}
\end{center}
\end{table}

\begin{table}  	%tbl:4.3
\begin{center}
\footnotesize
\begin{tabular}{llllll} \hline
\quad &&&& \\[-9pt]
&seed  &constants   &sequence   &label and symbol \\[2pt]  \hline
\quad &&&& \\[-9pt]
1. &(2, $-$1,2) &$\alpha = -1,\ \beta = 3$  
   &2, 2, $-$1, 2, 2, $-$1, 2, 2, $-$1, 2, 2, $-$1, 2, 2, {\ldots}    &[$-$1,2;1]   &${}^1(1,2)^1$ \\
2. &(3, $-$1, 2) &$\alpha =0,\ \beta  = 5$   
   &2, 6, 3, $-$1, 2, 6, 3, $-$1, 2, 6, $-$1, 2, 6,$-$1, {\ldots}     &[$-$1,2;2]   &${}^1(1,2)^2$ \\
3. &(14, -6, 15) &$\alpha =0,\ \beta = 29$   
   &14, -6, 15, 35, 14, -6, 15, 10, 35, 14, -6 {\ldots}               &---          &${}^2(5,7)^1$ \\
4. &(1, 1, 0) &$\alpha =0,\ \beta = 1$   
   &1, 1, 0, 0, 1, 1, 0, 0, 1, 1,0, 0, 1, 1, 0, 0, 1, 1,  {\ldots}    &[0,1;2]      &${}^1(1,1)^2$ \\
5. &(4, -1, 2) &$\alpha =1,\ \beta = 7$   
   &2, 10, 15, 12, 4, -1, 2, 10, 15, 15, 12, 4, -1 {\ldots}           &[$-$1,2;3]   &${}^1(1,2)^3$ \\
6. &(10, $-$6, 33) &$\alpha = 1,\ \beta = 49$  
   &33, 88, 104, 65, 10, $-$6, 33, 88, 104, 65, 10, $-$6, ...         &---          &${}^3(5,2)^1$  \\[3pt]
\hline
\end{tabular}
\caption{Examples of periodic sequences. Example 4 is known as \seqnum{A021913}.}
\label{tbl:4.3}
\end{center}
\end{table}

\clearpage

\begin{figure}[h] %================================
\[
\includegraphics[scale=.750]{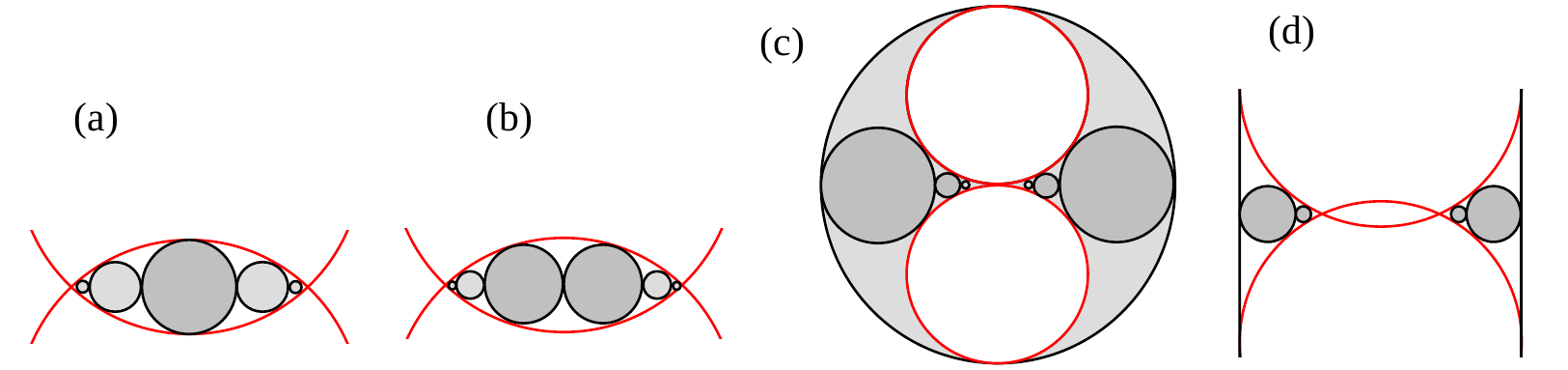} 
\]
\caption{Geometric representations (only general shape) for 
examples from Table \ref{tbl:4.1}: (a)~Examples 1--4, (b)~Examples 9--14, (c) Examples 5--8, (d) Examples 15--17.}
\label{fig:4.4}
\end{figure}
\begin{figure}[h] %================================
\[
\includegraphics[scale=.870]{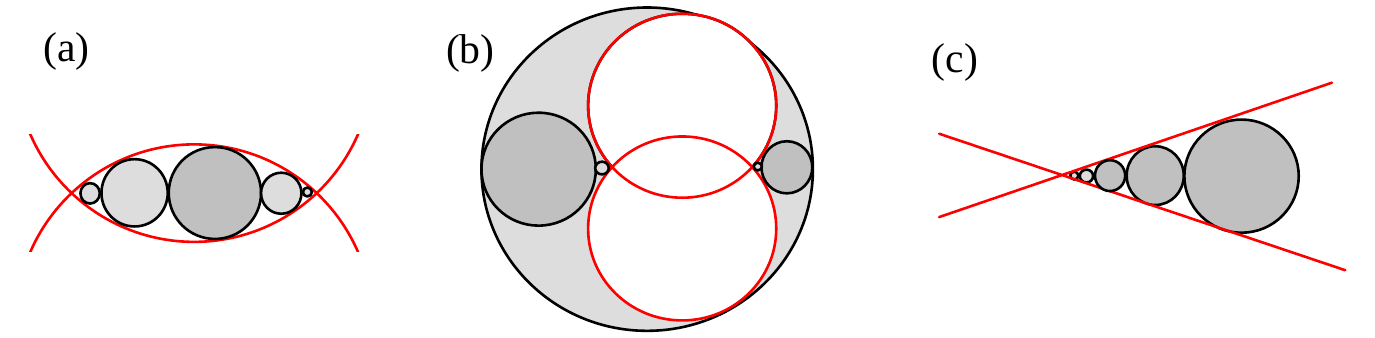} 
\]
\caption{Geometric representations for examples from Table \ref{tbl:4.2}: (a) Examples 1--2 \quad (b)~Example 3, (c) Example 4.}
\label{fig:4.5}
\end{figure}
\begin{figure}[h] %================================
\[
\includegraphics[scale=.66]{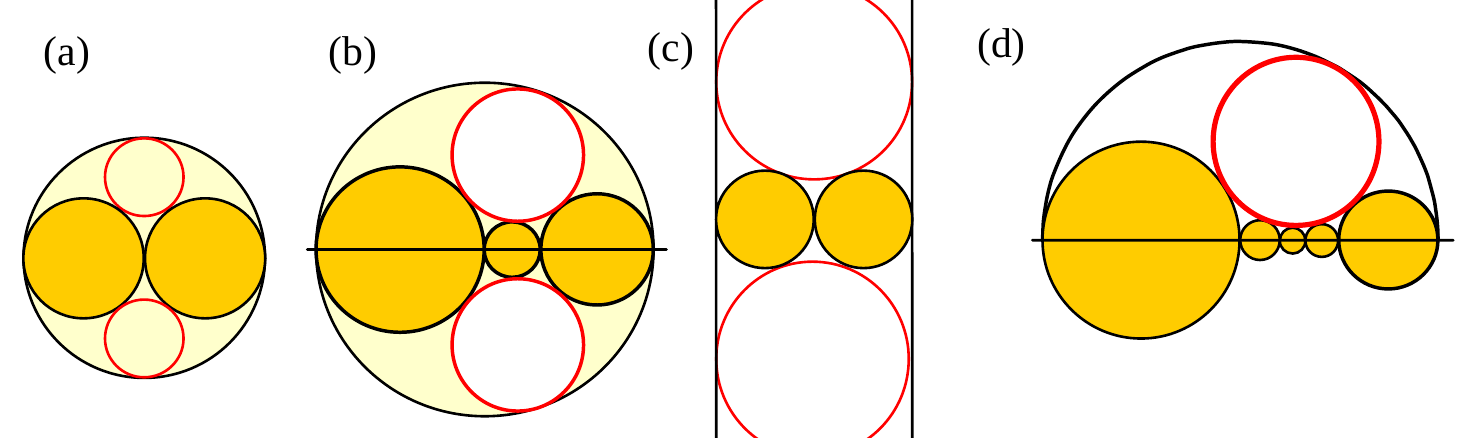} 
\]
\caption{Geometric representations for examples from Table \ref{tbl:4.3}: 
          (a) Example 1, (b) Example 2, (c) Example 3, (d) Example 4.}
\label{fig:4.6}
\end{figure}
\begin{table}[h!]  	%tbl:4.4
\begin{center}
\footnotesize
\begin{tabular}{llclcl} \hline
\quad &&&& \\[-9pt]
&constants  &OEIS  &sequence   &label and symbol \\[2pt]  \hline
\quad &&&& \\[-9pt]
1. &$\alpha = -3,\ \beta = 1$ &\seqnum{A001654} 
   &{\scriptsize $0, 1, -2, 6, -15, 40, -104, 273, -714, 1870, -4895, 12816, {\ldots}$} 
   &[0,1;--1]   &${}^{\phantom{-}1}(1,1)^{-1}$ \\
2. &$\alpha =-3,\ \beta = 5$ &\seqnum{A075269} 
   &{\scriptsize$2,-3,12,-28,77,-198,522,-1363,3572,-9348,{\ldots}$}                       
   &[2,2;1]   &${}^{-1}(1,2)^1$ \\
3. &$\alpha =-4,\ \beta = 3$ & --- 
   &{\scriptsize $1,-2,10,-35,133,-494,1846,-6887,25705,-95930, {\ldots}$}    
   &[1,1;--1]  &${}^{\phantom{-}2}(1,1)^{-1}$ \\
4. &$\alpha =-4,\ \beta = -1$   &\seqnum{A109437}
   &{\scriptsize $0,-1,3,-12,44,-165,615,-2296,8568,-31977,{\ldots}$} 
   &[0,--1;2]      &${}^{\phantom{-}2}(1,3)^{-1}$ \\
5. &$\alpha =-5,\ \beta = 1$   &\seqnum{A099025}
   &{\scriptsize $0,1,-4,20,-95,456,-2184,10465,-50140,240236,{\ldots}$}
   &[0,1;--3]      &${}^{-3}(1,1)^1$ \\
6. &$\alpha =-6,\ \beta = -4$   &\seqnum{A084159}
   &{\scriptsize $1,-3,21,-119,697,-4059,23661,-137903,803761,{\ldots}$ }           
   &[1,1;--2]      &${}^{-2}(1,1)^{-2}$ \\
6. &$\alpha = -6,\ \beta = 1$ &\seqnum{A084158}
   &{\scriptsize $0,1,-5,30,-174,1015,-5915,34476,-200940,{\ldots}$}  
   &[0,1;--4]      &${}^{-4}(1,1)^1$  \\[2pt]
\hline
\end{tabular}
\caption{{\small Examples of formal lens sequences}}
\label{tbl:4.4}
\end{center}
\end{table}

\clearpage

%%%%%%%%%%%%%%%%%%%%%%%%%%%%%%%%%%%%  END OF TABLES AND FIGURES %%%%%%%%%%%%%%

But the above types of sequences do not exhaust the possibilities, 
as this example shows:

\begin{center}
{\ldots} 2331, 407, 77, 21, 15, 35, 161, 897, 5187 {\ldots} ,
\end{center}

\noindent
which is a lens sequence with recurrence formula $b_{n} = 6 b_{n-1} - b_{n-2} - 34$. 
We will arrive at a general rule that produces all integer lens sequences 
and an improved version of the integrality criterion in the last section.

%%%%%%%%%%%%%%%%%%%%%%%%%%%%%%%%%%%%%%%%%%%%%%%%
\def\zzzz{
%-------------------------------------------------------------------------
\section{Binet-type formulas for lens sequences}  \label{S5}

The lens sequences may be generated by a Binet-type formula. 

\paragraph{Examples.}

\paragraph{1.} 
Vesica Piscis [\seqnum{A011922}]: \textbf{6 2 6 66 902 12546 174726 
{\ldots}}, $b_{n} = 14b_{n-1} - b_{n-2} - 16$
\[
  b_{n}  = \frac{4+\lambda^n+\bar{\lambda}^n}{3}
	=\frac{4+\kappa^{2n}+\bar {\kappa}^{2n}}{3} \hbox{ where }
	\lambda = 7+4\sqrt 3 =(2+\sqrt 3)^2=\kappa^{2 }\, .
\]
(see explanation below for $\bar\lambda$).

\paragraph{2.} Golden Vesica [\seqnum{A064170}]: \textbf{ {\ldots} 2 1 2 
10 35 442 }\textbf{\textit{{\ldots}}}, $b_{n} = 7b_{n-1} - b_{n-2} - 3$
\[
  b_{n}  = \frac{3+\lambda ^n+\bar {\lambda }^n}{5} \hbox{ where }
	\lambda = \frac{7+3\sqrt 5 }{2}=\varphi^{4}\, ,
\]
the fourth power of the golden ratio, is the characteristic constant of this 
sequence. The formula is ``centered'': for $n=0$ we have $b_{n} = 1$.

\paragraph{3.} Also non-symmetric lens sequences can be expressed 
by similar Binet-type formulas. For instance, the sequence extended 
from $(6, 2, 3)$:
\begin{center}
\textbf{{\ldots}2346 299 39 }\textbf{6 2 3}\textbf{ 15 110 858 6747{\ldots}}
\end{center}
with the recurrence formula $b_{n} = 8b_{n-1} - b_{n-2} - 7$ may be obtained from 
\[
  b_{n} = \frac{(25-3\sqrt{15}) \lambda^n+(25+3\sqrt{15})
	\bar {\lambda}^n}{60} + \frac{7}{6} \hbox{ where }
	\lambda  = 4 + \surd 15.
\]

Note that $\lambda $ as well as the other terms in the above formulas are 
elements of $\mathbb{Q}(\sqrt {\alpha ^2-4})$, the field of rational numbers 
extended by $p = \sqrt {\alpha ^2-4}$. The conjugation denoted by the bar is 
understood as a ``reflection'' of $p$.
For instance, in $\mathbb{Q}(\surd 5)$ we have 
$\overline {2+3\sqrt 5 } =2-3\sqrt 5$. Note that in each of the above 
examples $\vert \lambda \vert ^{2}=\lambda \bar {\lambda } = 1$, 
and thus $\bar {\lambda }=\lambda^{-1}$

\begin{theorem}  	\label{th:5.1} 
The lens sequence generated from a seed $(a,b,c)$ has the following Binet-like formula 
\begin{equation}    \label{eq:5.1}
  b_{n}=w\lambda^n+\bar {w}\bar {\lambda }^n+\gamma
\end{equation}
where 
\[
  \lambda = \frac{\alpha +\sqrt {\alpha ^2-4} }{2}
	\qquad
	\bar {\lambda } = \frac{\alpha -\sqrt {\alpha ^2-4} }{2}.
\]
and where
\[
  w = \frac{a-2b+c}{2(\alpha -2)}+\frac{c-a}{2(\alpha ^2-4)}
	\sqrt {\alpha ^2-4} , 
	\qquad 
	\gamma = \frac{-\beta }{\alpha -2}
\]
and $\bar {w}$ and $\bar {\lambda }$ denote conjugates of $w$ and $\lambda $ in 
$\mathbb{Q(}\sqrt {\alpha ^2-4} )$, respectively. 
In particular, $(a,b,c)=(b_{-1},b_0,b_{n+1})$.
\end{theorem}

\begin{proof} 
Define a new sequence whose entries are shifted by a constant, namely
$a_{n} = b_{n}+\beta/(\alpha -2)$.
The sequence $(a_{n})$ satisfies a homogeneous three-term recurrence formula:
$a_{n}=\alpha \, a_{n-1} - a_{n-2}$ which has a matrix formulation:
\[
  \left[ {{\begin{array}{c}
	a_{n+1} \\
	a_n     \\
	\end{array} }} \right] = 
  \left[ {{\begin{array}{cr}
	\alpha & {-1} \\
	1  & 0  \\
	\end{array} }} \right]
	\quad
	\left[ \begin{array}{c}
	a_n     \\
	a_{n-1} \\
	\end{array}  \right]
\]
that resolves to \ref{eq:5.1} by the standard procedure. 
\end{proof}

The corresponding characteristic polynomial 
$\lambda ^{2} - \alpha \lambda  + 1 = 0$ has two solutions
\[
  \lambda _{1} = \lambda = \frac{\alpha +\sqrt {\alpha ^2-4} }{2}, \quad   
  \lambda _{2} = \bar {\lambda } = \frac{\alpha -\sqrt {\alpha ^2-4} }{2}\,.
\]
Note $\lambda \bar {\lambda }=1$, that is $\bar {\lambda }=\lambda ^{-1}$. 
}%%%%%%%%%%%%%%%%%%%%%%%%%%%%%%%%%%%%%%%%%%%%%%%%%%%%%%%%%%%%%

%%%%%%%%%%%%%%%%%%%%%%%%%%%%%%%%%%%%%%%%%%%%%%%%%%%%%%%%%%%%%%%%%%%%%%
                                                                     %
\def\killx{  
\begin{proof} 
Given a sequence obeying an inhomogeneous three-term recurrence formula:
\[
  b_{n}=\alpha \, b_{n-1} - b_{n-2}+\beta 
\]
we may define a new sequence $(a_{n})$ whose entries are shifted by a constant, namely
\[
  a_{n} = b_{n}+\frac{\beta }{\alpha -2}\,.
\]
The sequence $(a_{n})$ satisfies a homogeneous three-term recurrence formula:
\[
  a_{n}=\alpha \, a_{n-1} - a_{n-2}\, ,
\]
which has a matrix formulation:
\[
  \left[ {{\begin{array}{c}
	{a_{n+1} }  \\
	{a_n }  \\
	\end{array} }} \right] = 
  \left[ {{\begin{array}{cr}
	\alpha & {-1} \\
	1  & 0  \\
	\end{array} }} \right]
	\quad
	\left[ {{\begin{array}{c}
	{a_n }  \\
	{a_{n-1} }  \\
	\end{array} }} \right]
\]
$(\det M = 1, \hbox{tr }M = \alpha)$. The corresponding characteristic polynomial 
$\lambda ^{2} - \alpha \lambda  + 1 = 0$ has two solutions
\[
  \lambda _{1} = \lambda = \frac{\alpha +\sqrt {\alpha ^2-4} }{2}, \quad   
  \lambda _{2} = \bar {\lambda } = \frac{\alpha -\sqrt {\alpha ^2-4} }{2}\,.
\]
Note $\lambda \bar {\lambda }=1$, that is $\bar {\lambda }=\lambda ^{-1}$. 
The corresponding eigenvectors, [$\lambda 1]^{T}$ and [$\bar {\lambda } 
1]^{T}$, provide the diagonalization of $M$:
\[
  M=\left[ {{\begin{array}{*{20}c}
	\lambda \hfill & {\bar {\lambda }} \hfill \\
	1 \hfill & 1 \hfill \\
	\end{array} }} \right]\left[ {{\begin{array}{*{20}c}
	\lambda \hfill & 0 \hfill \\
	0 \hfill & {\bar {\lambda }} \hfill \\
	\end{array} }} \right]\left[ {{\begin{array}{*{20}c}
	\lambda \hfill & {\bar {\lambda }} \hfill \\
	1 \hfill & 1 \hfill \\
	\end{array} }} \right]^{-1}
\]
The $n$-th power of $M$ is
\[
  M^{n}=\left[ {{\begin{array}{*{20}c}
	\lambda \hfill & {\bar {\lambda }} \hfill \\
	1 \hfill & 1 \hfill \\
	\end{array} }} \right]\left[ {{\begin{array}{*{20}c}
	{\lambda ^n} \hfill & 0 \hfill \\
	0 \hfill & {\bar {\lambda }^n} \hfill \\
	\end{array} }} \right]\left[ {{\begin{array}{*{20}c}
	\lambda \hfill & {\bar {\lambda }} \hfill \\
	1 \hfill & 1 \hfill \\
	\end{array} }} \right]^{-1}=\frac{1}{\lambda -\bar {\lambda }}\left[ 
	{{\begin{array}{*{20}c}
	\lambda \hfill & {\bar {\lambda }} \hfill \\
	1 \hfill & 1 \hfill \\
	\end{array} }} \right]\left[ {{\begin{array}{*{20}c}
	{\lambda ^n} \hfill & 0 \hfill \\
	0 \hfill & {\bar {\lambda }^n} \hfill \\
	\end{array} }} \right]\left[ {{\begin{array}{*{20}c}
	\ \ 1 \hfill & {-\bar {\lambda }} \hfill \\
	{-1} \hfill & \ \ \lambda \hfill \\
	\end{array} }} \right] .
\]
Acting with $M^n$ on the initial vector gives
\[
  \left[ {{\begin{array}{*{20}c}
	{a_{n+1} } \hfill \\
	{a_n } \hfill \\
	\end{array} }} \right] = M^{n}  
  \left[ {{\begin{array}{*{20}c}
	{a_1 } \hfill \\
	{a_0 } \hfill \\
	\end{array} }} \right] ={\ldots} = 
  \frac{1}{\lambda -\bar {\lambda}}\left[ {{\begin{array}{c}
	\ast  \\
	{(\lambda ^n-\bar {\lambda }^n)a_1 -(\lambda ^{n-1}-\bar {\lambda}^{n-1})a_0 }  \\
	\end{array} }} \right],
\]
where the entry ``*'' does not concern us. Thus the entries of the 
inhomogeneous three-term recurrence sequence $b_{n}=\alpha _{ 
}b_{n-1} - b_{n-2}+\beta $ are
\[
  b_{n} = \frac{(b_1 +z)-\bar {\lambda }(b_0 +z)}
	{\lambda -\bar {\lambda}}\lambda^n
	+\frac{\lambda (b_0 +z)-(b_1 +z)}{\lambda -\bar {\lambda }}
	\bar {\lambda }^n-z\,.
\]
Recall $\lambda -\bar {\lambda }=\sqrt {\alpha ^2-4} $, and $z =\beta 
/(\alpha -2)$. Thus the entries of the inhomogeneous three-term 
recurrence sequence, after rationalization of the denominators, are
\[
  b_{n} = w\lambda ^n+\bar {w}\bar {\lambda }^n-\tfrac{\beta }{\alpha -2}
\]
with 
\[
  w = \frac{b_1 -2b_0 +b_{-1} }{2(\alpha -2)}
	+\frac{b_1 -b_{-1}}{2(\alpha ^2-4)}\sqrt {\alpha ^2-4} \,.
\]
Renaming the terms $b_{-1} = a$, $b_{0} = b$, $b_{1} = c$, the proof is concluded. 
\end{proof}
}
                                                                 %
%%%%%%%%%%%%%%%%%%%%%%%%%%%%%%%%%%%%%%%%%%%%%%%%%%%%%%%%%%%%%%%%%%

\subsection*{Invariants}
The formulas for the coefficients $\alpha$ and $\beta$ in the recurrence 
formula \eqref{eq:1.1} may be represented diagrammatically as shown in 
Figure \ref{fig:4.2}, which exhibits the coefficients' algebraic ``structure'' 
(the mnemonic role aside). The dots on the line represent the consecutive terms 
of the sequence. The arcs represent products of the joined terms, and the 
position above/below the line position indicates their appearance 
in the numerator/denominator of the formula. Dotted lines are to be taken 
with the negative sign. %See Figure \ref{fig:3.2} for details.

\begin{figure}[h] %================================
\[
\includegraphics[scale=1]{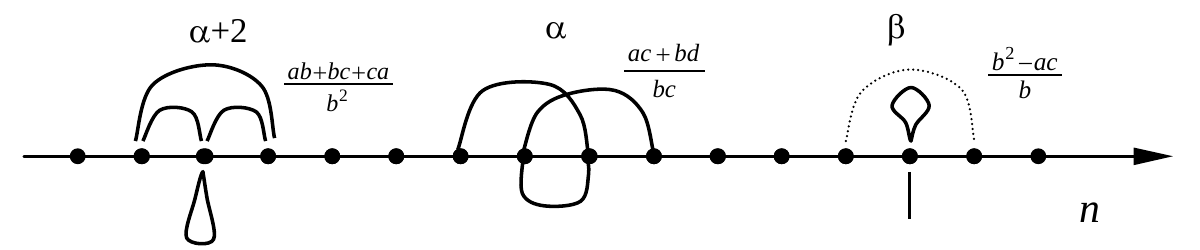} 
\]
\caption{Diagrammatic representation of the lens sequence invariants}
\label{fig:4.2}
\end{figure}

Since the formulas do not depend on the particular choice of the three 
seed circles, one may position the diagram at any place in the line/sequence. 
In this sense, \textit{$\alpha $} and \textit{$\beta $} represent \textbf{invariants} 
of the sequence with respect to translation along the sequence. 
But this also means that each of them gives rise to a new non-linear recurrence formula!

%------------------------------------------------------------
\subsection*{Additional remarks on the Binet-like formula}

Denoting ``jumps'' around the central element $b_{0}$ by 
$\Delta_{+} = b_{1} - b_{0}$ and by $\Delta_{ -} = b_{0} - b_{-1}$, 
we get a more suggestive form of term $w$ in the Binet-like formula
for lens sequences, namely:
\[
  w = \frac{\Delta _+ -\Delta _- }{2(\alpha -2)}+\frac{\Delta _+ 
	+\Delta_-}{\alpha ^2-4}\sqrt {\alpha ^2-4} \,.
\]
Expressing $\alpha $ in terms of $\lambda $ we also obtain
\[
  w = \frac{\Delta _+ -\Delta _- }{2(\lambda +\lambda ^{-1}-2)} 
	+ \frac{\Delta _+ +\Delta _- }{\sqrt {\lambda ^2+\lambda ^{-2}} }\,.
\]
Other representations of the formula for $w$ include: 
\[
\begin{aligned}
  w &= \left( {\frac{b_0 }{2}+\frac{\beta }{\alpha -2}} \right)
	+\frac{b_1 -b_{-1} }{2\alpha ^2-8}\sqrt {\alpha ^2-4} \, ,	\\
  w &= \frac{(\alpha +2)(b_1 -2b_0 +b_{-1} )\quad +\quad (b_1 -b_{-1} )
	\sqrt {\alpha ^2-4} }{2(\alpha ^2-4)}\,.
\end{aligned}
\]

Note that re-indexing the sequence so that another term of $(b_{n})$ becomes 
the central ``$b_{0}$'' will change the value of $w$ in the formula \eqref{eq:5.1}. 
\\

In order to better understand the situation, note the following simple 
general rule for this type of recurrences: 
\newline

\begin{theorem}		\label{th:5.3} 
Let a sequence $(x_{n})$ be given by the following Binet-like formula
\begin{equation}    \label{eq:5.4}
  x_{n} = a \omega^{n} + b \omega^{-n} + c
\end{equation}
for some constants $a$, $b$, $c$ and $\omega$. Then the sequence satisfies a 
non-homogeneous 3-term linear recurrence formula
\[
  x_{n+2} = (\omega + 1/\omega) x_{n+1} - x_{n} + c(2 - \omega - 1/\omega )\,.
\]
Moreover, if $\vert \omega \vert >1$ 
then $\mathop {\lim}\limits_{n\to \infty } \;\frac{x_{n+1} }{x_n } = \omega $, 
and for large $n$ we have $x_n\simeq a\omega ^n+c$.
\end{theorem}

\begin{proof} Direct. 
First note that \eqref{eq:5.4} for the ($n+2$)-nd term gives
\[
  x_{n+2} = a\omega ^{2}\omega ^{n}+b\omega ^{-2}\omega ^{-n}+c\,.
\]
Adding these two gives, after some simple algebraic operations:
\[
  x_{n}+x_{n+2} = (\omega + 1/\omega) x_{n+1}+c(2 - \omega - 1/\omega )
\]
which is equivalent to \eqref{eq:5.4}. 
\end{proof}

Looking at the above, one hardly escapes the thought that lens sequences 
could be ``explained'' in terms of Chebyshev polynomials. 
This path did not however return any deeper insight.
Instead, consider the following.

\newpage
%-------------------------------------------------------------------------
\section{Underground sequences}  \label{S8}

There are still some mysteries in the structure of lens sequences. 
One of the most remarkable properties is this: the entries of an integer 
lens sequence are products of consecutive pairs of a certain ``underlying'' integer sequence. 
Here is an example of a sequence of \textit{Vesica Piscis} 
(\seqnum{A011922}, see Example \ref{exm:1.1} in Section \ref{S1}):
\begin{figure}[h!] %================================
\[
\includegraphics[scale=1]{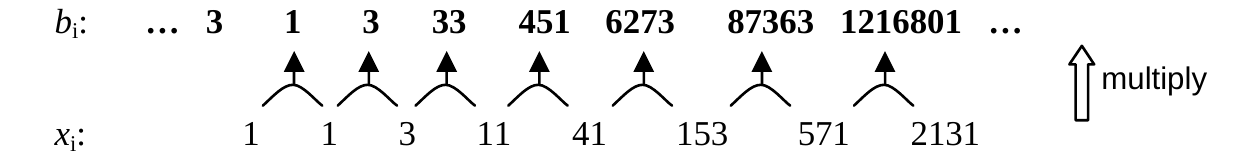} 
\]
\label{fig:8.1a}
\end{figure}
\vskip-.2in

\noindent
Thus the lens sequence may be represented as $b_{i}=f_{i-1}f_{i}$, for some  
integer sequence $(f_{i})$.
Sequence $(f_{i})$ will be called in this context the \textbf{underground} 
sequence of sequence $(b_{i})$  (Table 6.1 provides examples). 
We present the formalism of this amazing and unexpected property. 

Let us start with a general fact. 

\begin{theorem}  \label{th:fstep2} 
Any factorization $\{f_n\}$ of lens sequence in a sense that $b_n=f_{n-1}f_n$  
satisfies 3-term recurrence formula
\begin{equation} \label{eq:Lambda}
                           f_{n+2} + f_{n-2} = \alpha f_n
\end{equation}
\end{theorem}

\begin{proof} 
Starting with the expression for $\alpha$ given in Proposition \ref{prp:3.1}, we have
\[
  \alpha =  \frac{b_{n-1} }{b_n }+\frac{b_{n+2} }{b_{n+1} } 
	= \frac{f_{n-2} f_{n-1} }{f_{n-1} f_n }
	+\frac{f_{n+1} f_{n+2} }{f_n f_{n+1}} 
	= \frac{f_{n-2} }{f_n }+\frac{f_{n+2} }{f_n }
	= \frac{f_{n-2} +f_{n+2} }{f_n }
\]
\end{proof}

It should be borne in mind that this property is true for {\em any} 
---not necessarily integer--- factorization of $\{b_i\}$.
Such factorizations are easy to produce, e.g., set 
$f_0=1$, $f_1=b_1$, $f_2=b_2/b_1$, $f_3=b_3b_1/b_0$, $f_4=b_4b_2b_0/b_3b_1$, etc.
However:

\begin{theorem}[Factorization theorem] 	\label{th:ic2} 
Any integer lens sequence $(b_n)$ may be factored into an integer sequence $(f_n)$
so that $b_n=f_{n-1}f_n$.  
If the lens sequence is primitive, the factorization is ---up to a sign ---unique.   
Moreover, in such a case $|f_n|=\gcd(b_n,b_{n+1})$.
\end{theorem}

\begin{proof}
Assume that $(b_n)$ is a primitive lens sequence.
Consider three consecutive terms $(a,b,c)$ and define
$$
f_0 = \frac{a}{\gcd(a,b)}
\quad
f_1 = \gcd(a,b)
\quad
f_2 = \frac{b}{f_1} = \frac{b}{\gcd(a,b)}
\quad
f_3 = \frac{c}{f_2} = \frac{\gcd(a,b)c}{b}
$$
Clearly, $a=f_0f_1$, $b=f_1f_2$, and $c=f_2f_3$.
We need to show that these four terms are integers. 
Terms $f_0$, $f_1$, and $f_2$ are integer by definition. 
As to the last term, use the formula $\beta = \frac{b^2-ac}{b}$: 
$$
\beta\in\mathbb{Z} \quad\Rightarrow\quad
\frac{ac}{b}\in\mathbb{Z} \quad\Rightarrow\quad
\frac{\gcd(ab)c}{b}\in\mathbb{Z}
$$
hence $f_3$ is an integer.  Thus $f_0,f_1,f_2,f_3\in\mathbb{Z}$ and the 
integrality of the whole sequence $(f_n)$ follows immediately from \ref{eq:Lambda}.

As to uniqueness of factorization of a primitive lens sequence, assume 
{\it a contrario} that
two integer quadruples, $\{f_0,f_1,f_2,f_3\}$ and $\{g_0,g_1,g_2,g_3\}$  
are the initial terms of two different factorizations of $(b_n)$.  
Then $g_0/f_0 = p/q$ for some mutually prime $p,q\in\mathbb{Z}$.
At least one of $p$ and $q$ is not 1; assume that it is $p\not=1$.
Since $g_ig_{i+1}=f_if_{i+1}$, we must have 
$$
\{g_0,g_1,g_2,g_3\} 
    = \{{\frac{p}{q}}\,f_0, 
        {\frac{q}{p}}\,f_1,
        {\frac{p}{q}}\,f_2, 
        {\frac{q}{p}}\,f_3 \} 
\ \subset \mathbb{Z} \,.
$$
Thus $q|f_0$ and $q|f_2$ (because $\gcd(p,q)=1$). 
But this means that 
$q\;|\; a$ (since $a\!=\! f_0f_1$), \  
$q\;|\; b$ (since $b\!=\! f_1f_2$), \ and 
$q\;|\; c$ (since $c\!=\! f_2f_3)$, 
against the assumption of primitivity of the lens sequence.
\end{proof}

The 3-term recurrence \ref{eq:Lambda} for the underground sequence involves only $\alpha$.  
Another interesting non-linear 4-term recurrence involves only $\beta$:

\begin{proposition} 	\label{prp:u4term} 
Any underground sequence $\{f_i\}$ of a lens sequence $\{b_i\}$ satisfies
the following quadratic recurrence formula:
\begin{equation}    \label{eq:Delta}
\det \left[ {\begin{array}{*{20}c}
             f_{n} \hfill & f_{n+1} \hfill \\
             f_{n+2}   \hfill & f_{n+3}  \hfill \\
\end{array} } \right] \equiv f_{n+3}f_{n} - f_{n+1}f_{n+2} = -\beta.
\end{equation}
\end{proposition}

\begin{proof} 
For any $n$ we have
\[
\begin{aligned}
 %     \det \left[ {{\begin{array}{*{20}c}
%	    {f_{n}} \hfill & {f_{n+1}} \hfill \\
%            {f_{n+2}} \hfill & {f_{n+3}} \hfill \\
%                    \end{array} }} \right]
        f_n f_{n+3} - f_{n+1}f_{n+2} 
        &=\frac{f_n f_{n+1} f_{n+2} f_{n+3} }{f_{n+1} f_{n+2}} - f_{n+1}f_{n+2} \\
        &=\frac{ b_{n+1} b_{n+3} }{ b_{n+2} } - b_{n+2} 
        =\frac{ b_{n+1} b_{n+3} -b_{n+2}^2}{ b_{n+2} } 
        = - \beta \, \hfil
\end{aligned}
\]
\end{proof}

Note that 
not any initial quadruple $(f_1,f_2,f_3,f_4)$ leads via 
recurrence \ref{eq:Lambda} or \ref{eq:Delta}to a sequence that underlines a lens sequence.  
When do they?  First, we notice that 
the underground sequences of lens sequences have an interesting anatomy. 
It turns out that they are determined by three-term linear recurrences 
with \textit{variable} ``constants''. 
Here is the central theorem for the underground sequences:

\begin{theorem}[Underground Sequence Structure]  \label{th:8.1} 
(i) Let $k,s\in\mathbb{Z}$ be two constants.  Define a sequence $f$ by
\begin{equation}    \label{eq:8.1}
  f_{n}= \begin{cases}
	kf_{n-1} -f_{n-2} &\hbox{if  n is even}  \\
	sf_{n-1} -f_{n-2} &\hbox{if  n is odd.} 
	\end{cases} 
\end{equation}
with some arbitrary initial terms $f_0,f_1\in\mathbb{Z}$.
Define $b_{n}=f_{n-1}f_{n}$. Then $(b_{n})$ is a lens sequence. 
The constants of its recurrence formula
\begin{equation}    
  b_{n}=\alpha \, b_{n-1} - b_{n-2}+\beta
\end{equation}
are  
\begin{equation}    \label{eq:8.2}
  \begin{cases}
	\alpha = ks - 2 \\
	\beta = kf_{1}^{2} + sf_{0}^{2} - ks\,f_{0}f_{1}\, ,
	\end{cases}
\end{equation}
(ii) Every lens sequence is of such type.  
In particular, for a primitive lens sequence with a seed $(b_{-1},b_0,b_{1})=(a,b,c)$ 
the underground sequence is defined
\[
\begin{alignedat}{2}
  f_0 &= \gcd(a,b),         &\qquad f_1 &= \gcd(b,c), \\
  s   &= \frac{a+b}{f_0^2}, &\qquad   k &= \frac{b+c}{f_1^2}.
\end{alignedat}
\]
\end{theorem}

\begin{proof}
We start with part (ii).  Let $(f_i)$ be a sequence defined by \ref{eq:8.1}.
First, we shall show that the following expression
\begin{equation}    \label{eq:Delta}
  \Delta _{n}=\det \left[ {{\begin{array}{*{20}c}
  {f_{n-3} } \hfill & {f_{n-2} } \hfill \\
  {f_{n-1} } \hfill & {f_n } \hfill 
\end{array} }} \right]
\end{equation}
is an invariant of a sequence \eqref{eq:8.1}, that is it does not depend on $n$.
Indeed, let $p$ denote $k$ or $s$, depending on whether $n$ is even or odd 
(it will not matter!). Then
\[
\begin{alignedat}{2}
  \Delta _{n}&=\det \left[ {{\begin{array}{*{20}c}
 {f_{n-3} } \hfill & {f_{n-2} } \hfill \\
 {f_{n-1} } \hfill & {f_n } \hfill \\
\end{array} }} \right] & \ &=\det \left[ {{\begin{array}{*{20}c}
 {f_{n-3} } \hfill & {pf_{n-3} -f_{n-4} } \hfill \\
 {f_{n-1} } \hfill & {pf_{n-1} -f_{n-2} } \hfill \\
\end{array} }} \right] \\
&= \det \left[ {{\begin{array}{*{20}c}
 {f_{n-3} } \hfill & {-f_{n-4} } \hfill \\
 {f_{n-1} } \hfill & {-f_{n-2} } \hfill \\
\end{array} }} \right] & \ &=\det \left[ {{\begin{array}{*{20}c}
 {f_{n-2} } \hfill & {f_{n-1} } \hfill \\
 {f_{n-4} } \hfill & {f_{n-3} } \hfill \\
\end{array} }} \right] \quad =\Delta _{n-1}\,.
	\\
\end{alignedat}	
\]
Thus, by induction, the value of $\Delta _{n}$ does not depend on $n$. 
Hence it may be brought down to the first two terms of 
the sequence and, after simple substitution, shown to be 
\[
  \Delta _{n} = kf_{1}^{2} + sf_{0}^{2} - ks\,f_{0}f_{1}\, .
\]

As to the recurrence formula, calculate the sum 
$b_{n-1}+b_{n+1}$ :
\begin{equation}    \label{eq:8.5}
\begin{aligned}
  b_{n+1}+b_{n-1} &= f_n f_{n+1} + f_{n-2} f_{n-1} 
	= (af_{n-1} - f_{n-2})(bf_n - f_{n-1}) + f_{n-2} f_{n-1} \\
   &= ks\,f_{n-1} f_n - kf_{n-1}^2 - sf_n f_{n-2} 
	+ f_{n-1} f_{n-2} + f_{n-1} f_{n-2}  \\
   &= ks\,f_{n-1} f_n - f_{n-1} (kf_{n-1} - f_{n-2})
	- f_{n-2} (sf_n - f_{n-1}) 	\\
   &= ks\,f_{n-1} f_n - f_{n-1} f_n - f_{n-2} f_{n+1}  \\
   &= ks\,f_{n-1} f_n - 2f_{n-1} f_n + f_{n-1} f_n 
	- f_{n-2} f_{n+1} 	\\
   &= (ks-2)\,f_{n-1} f_n + (f_{n-1} f_n - f_{n-2} f_{n+1})  \\
   &= (ks-2)\,b_n -\Delta _n \,.
\end{aligned}
\end{equation}
Now, rename $ks - 2 = \alpha$, and  
$\Delta _{n} =-\beta$ (as it does not depend on $n$). 
Then \eqref{eq:8.5} becomes an inhomogeneous three-term 
recurrence formula 
\[
  b_{n+1} + b_{n-1} = \alpha \,b_n +\beta \,.
\]
Finally, we need to show that $(b_n)$ is actually a lens system sequence. 
As for $\alpha$, we need simply to show
\[
  \alpha = \frac{b_{n-1} }{b_n }+\frac{b_{n+2} }{b_{n+1} }
\]
(see Proposition \ref{prp:3.1}). Consider the right hand side:
\[
  \frac{b_{n-1} }{b_n }+\frac{b_{n+2} }{b_{n+1} } 
	= \frac{f_{n-2} f_{n-1} }{f_{n-1} f_n }
	+\frac{f_{n+1} f_{n+2} }{f_n f_{n+1}} 
	= \frac{f_{n-2} }{f_n }+\frac{f_{n+2} }{f_n }
	= \frac{f_{n-2} +f_{n+2} }{f_n }
	= \alpha \,,
\]
where the last step is true because of \eqref{eq:8.1}. Indeed, assume $n$ is even:
\[
\begin{aligned}
  \frac{f_{n-2} +f_{n+2} }{f_n }
	&= \frac{f_{n-2} +kf_{n+1} -f_n }{f_n }	\\
	&= \frac{f_{n-2} +k(sf_n -f_{n-1} )-f_n }{f_n } 	\\
	&= \frac{ksf_n -kf_{n-1} -f_{n-2} -f_n }{f_n }	 	\\
	&= \frac{ksf_n -f_n -f_n }{f_n }	\\
	&= ks-2=\alpha \,.
\end{aligned}
\]
As to the formula for $\beta$ in terms of a seed, 
follow similar calculations as in the proof of Proposition \ref{prp:u4term}:
\[
  \Delta _{n}=\det \left[ {{\begin{array}{*{20}c}
	{f_0 } \hfill & {f_1 } \hfill \\
	{f_2 } \hfill & {f_3 } \hfill \\
  \end{array} }} \right]
        =  f_0 f_3 -f_1 f_2 
	= \frac{f_0 f_1 f_2 f_3 - (f_1 f_2)^2 }{f_1 f_2 }  
        = \frac{b_1 b_3 - b_2^2}{b_2 }  
	= - \beta \,,
\]
This ends the proof of (i).
The proof of (ii) follows easily.
Let $(a,b,c)$ be a primitive seed of an integer sequence.
Define a quadruple of numbers $(f_0,f_1,f_2,f_3,f_4)$
by setting $f_1 = \gcd(a,b)$,  $f_2 = \gcd(b,c)$,
and $f_0 = a/f_1$, and $f_4=c/f_2$.
Then $k=(f_0+f_2)/f_1$ and $s=(f_1+f_3)/f_2$.          
\end{proof}

\begin{remark}
Note that $\Delta _{i}$ may be viewed as a quadratic form given by the matrix 
\[
  G=\left[ {{\begin{array}{*{20}c}
	k & ks  \\
	0 & s  \\
  \end{array} }} \right]
\]
evaluated on the vector $\mathbf{v} = [ f_{0}, f_{1}]^{T}$, i.e., 
$\Delta _{n} = \mathbf{v}^{T}G\mathbf{v}$.
In particular, vectors $(f_i,f_{i+1})\in\mathbb{R}$ 
all stay on a quadratic defined by $G$.
\end{remark}

As a corollary of the above theorem, we arrive at a simple criterion on whether 
a triple of integers is a good candidate for an integer lens sequence,
improving that of Proposition \ref{prp:ic}:

\begin{theorem}[Integrality Criterion 2]  \label{th:ic2} 
A triple $(a,b,c)\subset\mathbb{Z}$ is a seed of an integer lens sequence 
iff 
$$
\begin{aligned}
 (i)    &\ \gcd(a,b)\gcd(b,c)=b\gcd(a,b,c)\\
 (ii)   &\ (\gcd(a,b))^2 \ \ \hbox{\rm divides } \ (a+b)\gcd(a,b,c)\\
 (iii)  &\ (\gcd(b,c))^2 \ \ \hbox{\rm divides } \ (b+c)\gcd(a,b,c)
\end{aligned}
$$
\end{theorem}

\begin{figure}[h] %====================
\[
\includegraphics[scale=1.1]{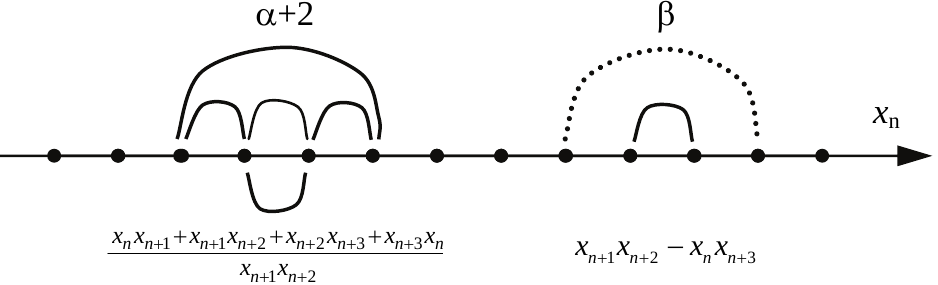} 
\]
\caption{Diagrammatic representation of the lens sequence 
invariants, calculated from the underground sequence (cf. Fig.\ \ref{fig:3.2})}
\label{fig:8-1b}
\end{figure}

The above result leads to another property of lens sequences:

\begin{corollary} 	\label{cor:8.3} 
The sum of any pair of consecutive terms of a lens sequence is 
a multiple of a square, namely:
\[
  b_{n}+b_{n+1}=\left\{ {{\begin{array}{*{20}c}
	{kf_n^2 \quad \mbox{if }n\mbox{ is even}} \hfill \\
	{sf_n^2 \quad \mbox{if }n\mbox{ is odd}\;} \hfill \\
  \end{array} }} \right.=\left\{ {{\begin{array}{*{20}c}
	{k\cdot \mbox{square}\quad \mbox{if }n\mbox{ is even}} 
		\hfill \\
	{s\cdot \mbox{square}\quad \mbox{if }n\mbox{ is odd.}\;}
		\hfill \\
  \end{array} }} \right.
\]
\end{corollary}

\begin{proof} Elementary:  $b_{n}+b_{n+1}=f_n f_{n-1} +f_{n+1} f_n = 
f_n (f_{n+1} +f_{n-1} )=f_n pf_n$, where $p$ stands for $k$ or $s$, 
depending on the parity of $n$. 
\end{proof}

For example, the sums of two consecutive entries of \seqnum{A011922} are perfect squares:
\[
\includegraphics[scale=.9]{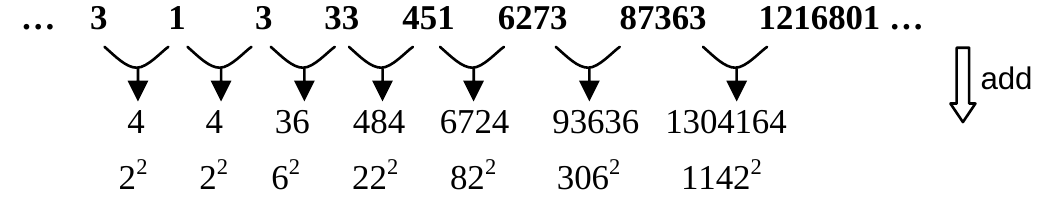} 
\]

The sequence \seqnum{A101265} (of label [1,1;3]) gives:
\[
\includegraphics[scale=.9]{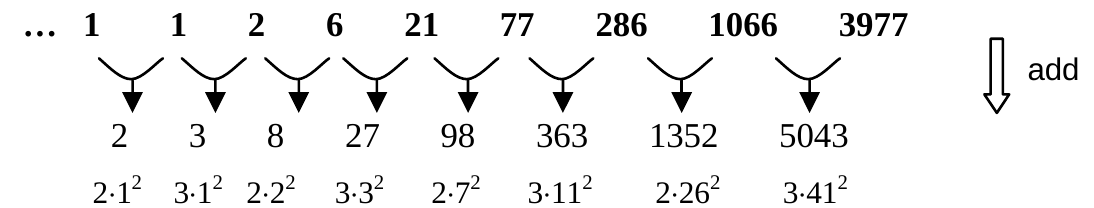} 
\]

%------------------------------------------------------------
\subsection*{Generating lens sequences}

Let us return to the question of generating integer lens sequences. In Section \ref{S2}, we considered triplets of the form $(a,b;k)$ that label a large family of lens sequences (see \eqref{eq:3.1}), but not all of them. The existence of underground sequences allows one to label all lens sequences.

\begin{definition} 	\label{def:Delta} 
A \textbf{symbol} of a lens sequence is the quadruple $^{s}(p, q)^{k}$ which defines the underground 
sequence $(f_{i})$ with $f_{0} = p$, $f_{1} = q$, and with constant $s$ and $k$ in \eqref{eq:8.1}, 
and therefore defines the corresponding lens sequence $(b_{i})$, namely, $b_{i}=f_{i-1}f_{i}$. 
More directly, symbol $^{s}(p, q)^{k}$ defines a lens sequence via its 
seed $(a,b,c) = ((sp - q)p, pq, q(kq - p))$. 
\end{definition}

Any integer quadruple $^{s}(p, q)^{k}$ leads to an integer lens sequence. And 
vice versa, given a seed of a primitive sequence $(a,b,c)$, we easily 
reproduce the symbol: 
\[
\begin{alignedat}{2}
  p &= \gcd(a,b), &\qquad k &= \frac{b+c}{q^2}, \\
  q &= \gcd(b,c), &\qquad s &= \frac{a+b}{p^2}.
\end{alignedat}
\]

\begin{proposition} 	\label{prp:8.5} 
The lens sequence generated by $^{s}(p, q)^{k}$ is primitive if and only if 
\begin{equation}    \label{eq:8.6}
  \gcd(p,q) = \gcd(p,k) = \gcd(s,q) = 1.
\end{equation}
\end{proposition}

\begin{proof} 
Write the ``central'' four terms of the underground sequences and the corresponding lens sequence: 
\[
\begin{aligned}
  (f_{i}): \qquad &{\ldots}, f_{-1} = (sp - q), 
	f_{0} = p, f_{1} = q, f_{2} = (kq - p), {\ldots}  \\
  (b_{i}): \qquad &{\ldots} , a = (sp - q)p, b = pq, 
	c = q(kq - p), {\ldots}
\end{aligned}
\]
For $(b_{i})$ to be primitive we must have $\gcd(a,b,c)=1$, which implies \eqref{eq:8.6}. 
\end{proof}

\begin{corollary} 	\label{cor:8.6} 
If the underground sequence $(f_i)$ contains $p=\pm1$, then the corresponding lens sequence $(b_i)$ admits label $(a,b;k)$. 
More precisely:
\[
\begin{gathered}
  (a, b; k) \hbox{ corresponds to the symbol } 	^{a+b}(1, b)^{k} \,;	\\
  \hbox{the symbol }  ^{s}(1, b)^{k} \hbox{ corresponds to } 	(s-b, b; k) \,.
\end{gathered}
\]
\end{corollary}

\begin{remark} 		\label{rem:8.6}
To use this generator of sequences as a unique \textit{label} system for lens sequences, one would have to remove the ambiguity of the choice of the initial terms. We may demand that, say, $p = f_{0}$ has the smallest absolute value among $(f_{i})$ and that $\vert f_{-1}\vert > f_{0}\leqslant f_{1}$.
\end{remark}

\paragraph{Remark on diagrammatic use of symbols.} 
The first of the following two diagrams means that 13 is obtained as $13 = 5\times 3 - 2$. 
The second represents equation $2 = 5\times 3 - 2$:
\[
\includegraphics[scale=.9]{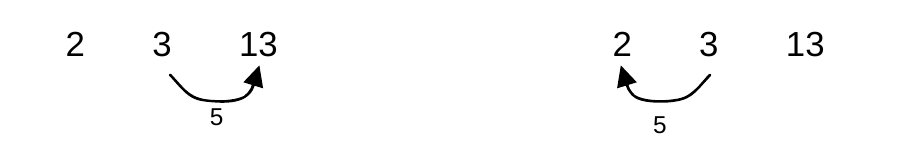} \label{fig:8-1e}
\]
Now, the symbol $^{1}(2,3)^{5}$ may be graphically developed into an 
underground sequence $(f_{i})$, and consequently into a lens sequence 
$(b_{i})$, in the following way: 
\[
\includegraphics[scale=.9]{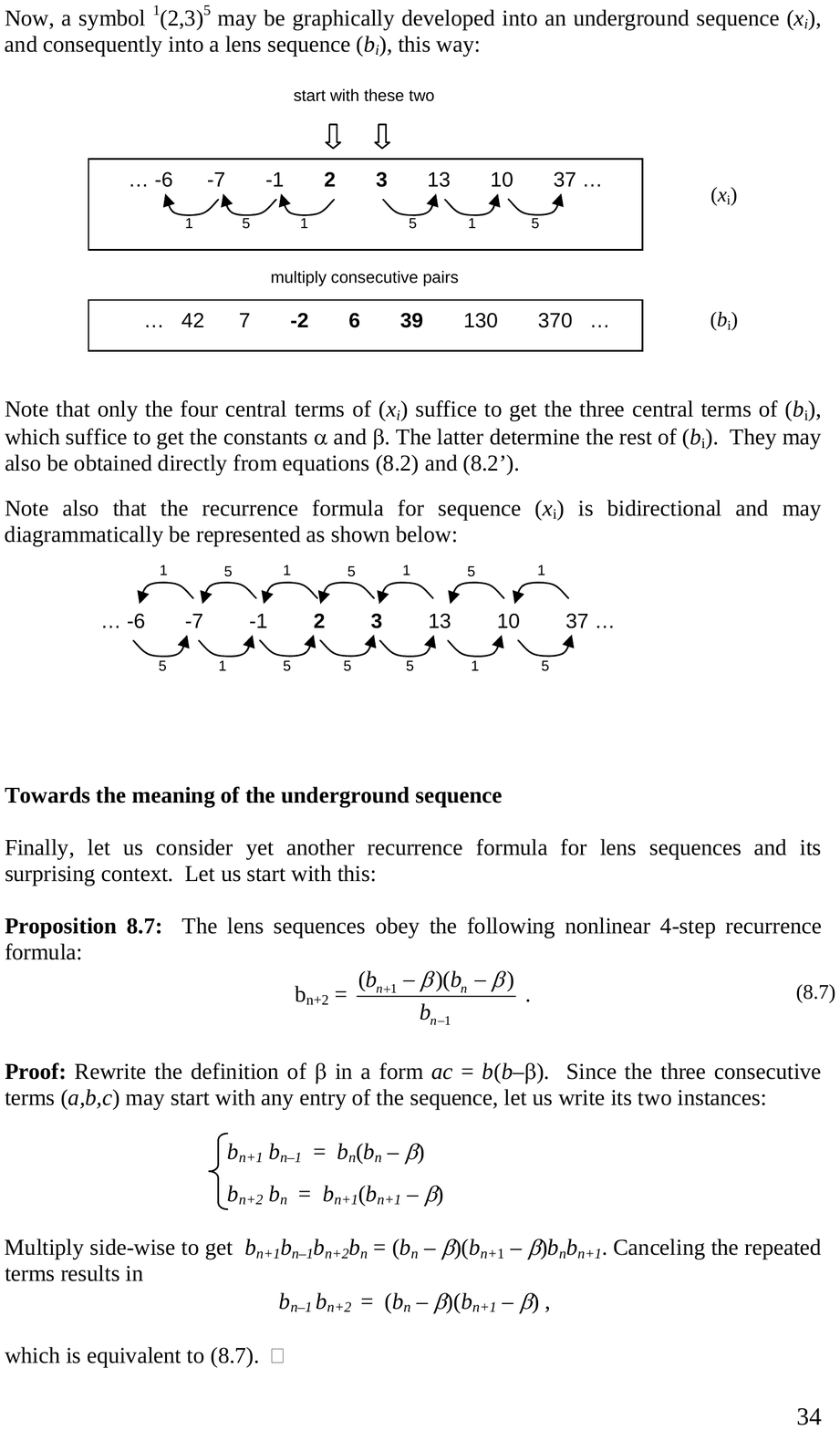} \label{fig:8-1f}
\]
Note that the four central terms of $(f_{i})$ suffice to generate the three central 
terms of $(b_{i})$, which yield the constants $\alpha $ and $\beta $. 
The recurrence formula for sequence $(f_{i})$ is bilateral 
and may be represented diagrammatically as shown below:
\[
\includegraphics[scale=.9]{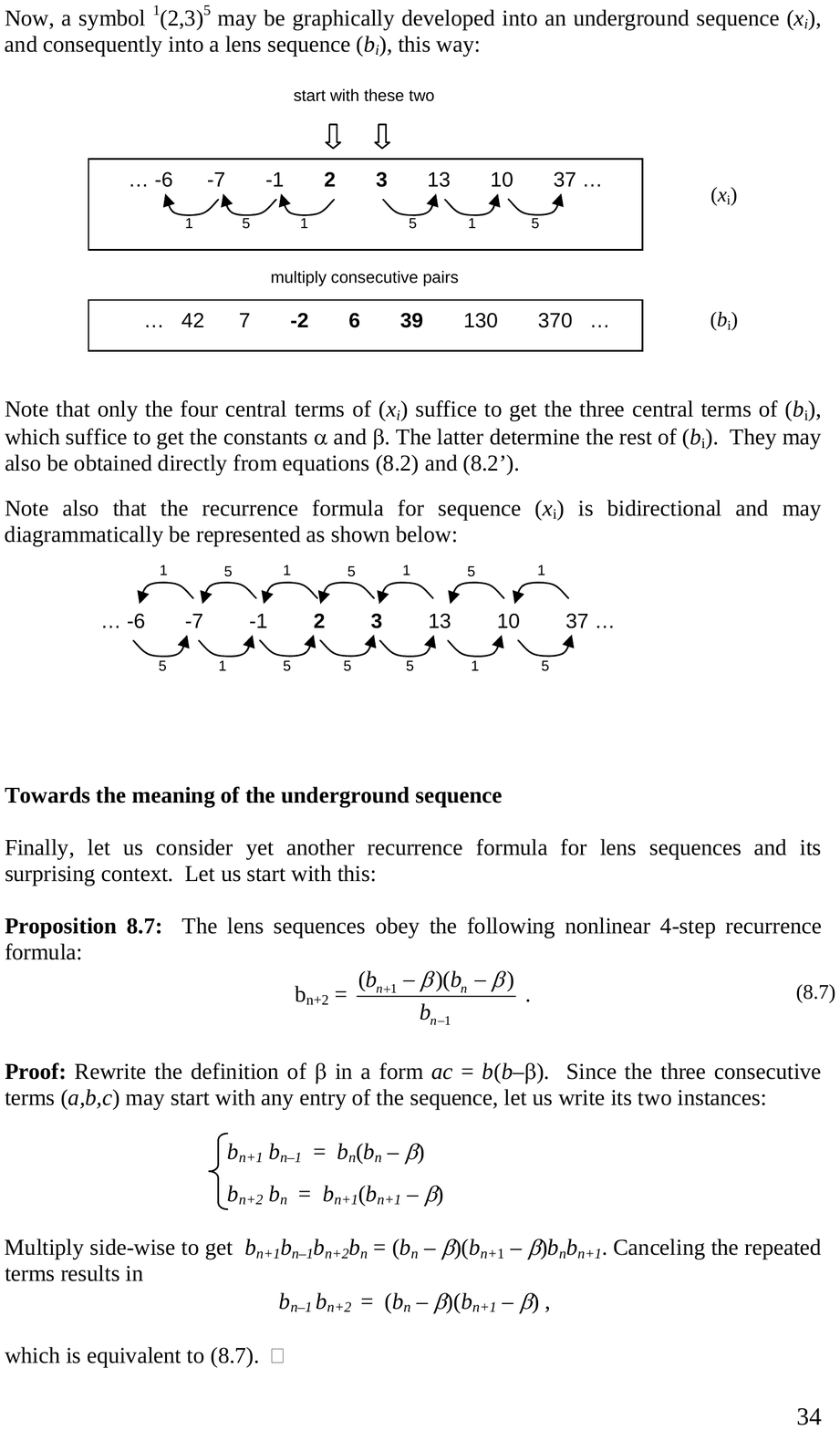} \label{fig:8-1g}
\]

%======================= table =========================================
\begin{table}[h]  	%tbl:underground
\footnotesize
\hrule
\vskip2ex
{\bf Example 1} (Vesica Piscis): 
$\mathbf{b} =$ \seqnum{A011922} \quad $\mathbf{f} =$ \seqnum{A001835} = \seqnum{ A079935}
\\
$\{b_i\} = (..., 1, 3, 33, 451, 6273, 87363,\ldots)$, \quad $b_{n+1}=14b_n-b_{n-1}-8$
\\
$\{f_i\} = (..., 1, 3, 11, 41, 153, 571,\dots)$, \quad  $^4(1,1)^4$ 
\quad (number of domino packings in a ($3\times 2n)$ rectangle)
\vskip.08in 
\noindent
{\bf Example 2} (Golden Vesica): 
$\mathbf{b} =$ \seqnum{A064170} \quad $\mathbf{f} =$ \seqnum{A001519} 
\\
$\{b_i\} = (..., 1, 2, 10, 65, 442, 30,26, 20737,\ldots)$ \quad $b_{n+1}=7b_n-b_{n-1}-3$
\\
$\{f_i\} = (..., 1, 2, 5, 13, 34, 89, 233, \dots)$, \quad $^3(1,1)^3$ \quad
(odd Fibonacci numbers)  

\vskip.08in 
\noindent
{\bf Example 3}:  $\mathbf{b} =$ \seqnum{A000466}  \quad $\mathbf{f} =$ \seqnum{A005408}\\
$\{b_i\} = (..., -1, 3, 15, 35, 63, 99, 143, 195, 255,\ldots)$ \quad $b_{n+1}=2b_n-b_{n-1}+8$\\
$\{f_i\} = (...-1, 1, 3, 5, 7, 9, 11, 13, 15, \dots)$, \quad   $^2(-1,1)^2$ 
\quad (odd numbers)

\vskip.08in 
\noindent
{\bf Example 4}:  $\mathbf{b} =$ \seqnum{A081078}  \quad $\mathbf{f} =$ \seqnum{A002878} \\
$\{b_i\} = (..., -1, 4, 44, 319, 2204,\ldots)$ \quad $b_{n+1}=7b_n-b_{n-1}+15$\\
$\{f_i\} = (...-1, 1, 4, 11, 29, 76, 199, 521, 1364 \dots)$ , \quad   $^3(-1,1)^3$
\quad (odd Lucas numbers)

\vskip.08in 
\noindent
{\bf Example 5}:  $\mathbf{b} =$ \seqnum{A005247}  \quad $\mathbf{f} =$ \seqnum{A005247}\\
$\{b_i\} = (...,2, 2, 3, 6, 14, 35, 90, 234, 611, 1598, \ldots)$ \quad $b_{n+1}=3b_n-b_{n-1}-1$\\
$\{f_i\} = (...,2, 1, 3, 2, 7, 5, 18, 13, 47, 34, \dots)$, , \quad   $^1(2,1)^5$ \\
(even Lucas numbers interlaced with odd Fibonacci numbers)
\vskip.08in 
\noindent
{\bf Example 6}:  $\mathbf{b} =$ \seqnum{A027941}  \quad $\mathbf{f} =$ \seqnum{A005013}\\
$\{b_i\} = (..., 0, 1, 4, 12, 33, 88, 232, 609, \ldots)$ \quad $b_{n+1}=3b_n-b_{n-1}+1$\\
$\{f_i\} = (..., 0, 1, 1, 4, 3, 11, 8, 29, 21, 76, \dots)$, , \quad  $^1(1,1)^5$ \\
(odd Lucas numbers interlaced with even Fibonacci numbers)
\vskip.08in 
\noindent
{\bf Example 7}:  $\mathbf{b} =$\seqnum{A000217} \quad $\mathbf{f} =$ \seqnum{A026741} \\
$\{b_i\} = (...,0, 1, 3, 6, 10, 15, 21, 28 \ldots)$ \quad $b_{n+1}=2b_n-b_{n-1}+1$ \\
$\{f_i\} = (...,0, 1, 1, 3, 2, 5, 3, 7, 4, 9, \dots)$ , \quad  $^4(0,1)^1$ \\
(the sequence of natural numbers interlaced with odd natural numbers):  
$
f_n=\begin{cases}
           n   & \hbox{if } n \hbox{ is odd} \cr
           n/2 & \hbox{if } n \hbox{ is even}
     \end{cases}
$
\vskip.08in
\hrule
\caption{Examples of underground sequences $\mathbf{f}$ for 
         some integer lens sequences $\mathbf{b}$.}
\label{tbl:underground}
\end{table}

%------------------------------------------------------------
\subsection*{Towards the meaning of the underground sequence}

Finally, let us consider yet another recurrence formula for lens sequences and its surprising context. Let us start with this:

\begin{proposition} 	\label{prp:8.7} 
Lens sequences obey the following nonlinear 4-step recurrence formula: 
\begin{equation}    \label{eq:8.7}
    b_{n+2}=\frac{(b_{n+1} -\beta )(b_n -\beta )}{b_{n-1} }\, .
\end{equation}
\end{proposition}

\begin{proof} Rewrite the definition of $\beta $ in the form $ac = b(b-\beta)$. 
Since the three consecutive terms $(a,b,c)$ may start with any entry of the sequence, 
let us write its two instances:
\[
  \begin{cases}
	b_{n+1} b_{n-1}=b_{n}(b_{n} - \beta) \\
	b_{n+2} b_{n}=b_{n+1}(b_{n+1} - \beta)
   \end{cases}
\]
Multiply side-wise to get $b_{n+1}b_{n-1}b_{n+2}b_{n} = (b_{n} - 
\beta)(b_{n+1} - \beta )b_{n}b_{n+1}$. Canceling the repeated terms results in
\[
  b_{n-1}b_{n+2} = (b_{n} - \beta)(b_{n+1} - \beta) ,
\]
which is equivalent to \eqref{eq:8.7}. 
\end{proof}

\begin{corollary} 	\label{cor:8.8} 
A lens sequence satisfies the following identity:
\[
\det \left[ {{\begin{array}{*{20}c}
 {b_{n+1} -\beta }  & {b_{n+2} }  \\
 {b_{n-1} }  & {b_n -\beta }  \\
\end{array} }} \right] = 0\, .
\]
\end{corollary}

This leads to yet another implication. The above determinant \eqref{eq:8.7} may be 
written in the form:
\[
\det \left( {\left[ {{\begin{array}{*{20}c}
 {b_{n+1} } \hfill & {b_{n+2} }  \\
 {b_{n-1} } \hfill & {b_n }  \\
\end{array} }} \right]-\left[ {{\begin{array}{*{20}c}
 \beta \hfill & 0 \hfill \\
 0 \hfill & \beta \hfill \\
\end{array} }} \right]} \right)=0\quad ,
\]
which looks like a characteristic equation with $\beta $ playing the role of the eigenvalue. Note 
that it does not depend on $n$. What are the corresponding eigenvectors?

\begin{theorem} 	\label{th:8.9} 
If $(f_{i})$ is the underground sequence of a lens sequence $(b_{i})$, then the following ``eigen-equation'' holds:
\[
  \left[ {{\begin{array}{*{20}c}
 {b_{n+1} } \hfill & {b_{n+2} } \hfill \\
 {b_{n-1} } \hfill & {b_n } \hfill \\
\end{array} }} \right]
\left[ {{\begin{array}{*{20}c}
 {f_{n+1} } \\
 {-f_{n-1} } \\
\end{array} }} \right]\quad = \quad \beta \left[ {{\begin{array}{*{20}c}
 {f_{n+1} } \\
 {-f_{n-1} } \\
\end{array} }} \right]\, .
\]
\end{theorem}

\begin{proof} Write $b$'s in terms of $f$'s and use the definition of $\beta $ written 
in terms of $f$'s. 
\end{proof}

\vskip.2in

%-------------------------------------------------------------------------
\section{Summary}	\label{S9}

A lens sequence is an integer sequence $(b_{i})$ that satisfies two 
conditions:
\[
\begin{alignedat}{3}
  \hbox{(i)}&&\quad &b_{n} = \alpha b_{n-1}-b_{n-2}+\beta 
	&\hbox to .5in{} &\hbox{[recurrence formula]} 	\\
  \hbox{(ii)}&&\quad &a^{2}+b^{2} = \alpha \,ab+\beta(a+b)
	&\hbox to .5in{} &\hbox{[compatibility relation]}
\end{alignedat}
\]
where $\alpha $ and $\beta$ are constants and $a$ and $b$ are any two
consecutive terms of $(b_{i})$. These two conditions assure that 
the sequence has a geometric realization in terms of the curvatures 
of a chain of circles inscribed in a symmetric lens 
(the space of the overlap of the interiors or exteriors of two congruent circles). 
The sequence constants may be viewed as invariants of a process $i \to b_{i}$. 
They may be calculated from a \textit{seed}, i.e., any three consecutive 
sequence terms $(a, b, c)$: 
\[
  \alpha = \frac{ab+bc+ca}{b^2}-1 	
	\hbox{ and }
  \beta = \frac{b^2-ac}{b}
\]
or, for $\alpha $, alternatively 
\[
  \alpha = \frac{b_{n-1}}{b_n}+\frac{b_{n+2}}{b_{n+1}}\, .
\]
Other, nonlinear, recurrence formulas for the lens sequence include:
\[
\begin{tabular}{lll}
  \lbrack two-step formula] 
	&\multicolumn{2}{l}{$2b_{n+1}= b_n \alpha 
                         +\beta\pm \sqrt {(\alpha^2-4)\;b_n^2
                         +2(\alpha+2)\beta \;b_n +\beta^2}$}	\\ \\
  \lbrack three-step formula] 
	&$b_{n+1}b_{n} + b_{n+1}b_{n-1}+b_{n} b_{n-1}  =(\alpha +1)b_{n}^2$ 
		\hbox to .3in{}
	&(only $\alpha$)	\\ \\
  \lbrack three-step formula] \hbox to .3in{} 
	&$ b_{n}b_{n} - b_{n+1}b_{n-1}  =   \beta b_{n} $
	&(only $\beta$)	\\ \\
  \lbrack four-step formula] 
	&$b_{n+2}b_{n-1} =(b_{n+1} -\beta )(b_n -\beta)$ \\ \\
  \lbrack four-step formula] 
	&$b_{n+1}b_{n-1} + b_{n}b_{n-2} =  \alpha b_n b_{n-1} $	\\ 
\end{tabular}
\]

The sequence constants have a geometric meaning: 
$\alpha$ codes the angle under which the circles forming the lens intersect 
(if they do), or, more generally, the Pedoe product of the lens circles. 
The value of $\beta $ reflects the size of the system. 
There are two basic properties determined by geometry: 
(a) the sum of the inverses is determined by the length of the lens, 
and (b) the limit of the ratio of consecutive terms is determined by 
the aforementioned lens angle: 
\[
  \sum\limits_n {\frac{2}{b_n }} 
	= \frac{\sqrt {\alpha ^2-4}}{-\beta } \quad \hbox{ and } \quad
	\lim_{i\to\infty} \frac{b_{i+1}}{b_i}
	=\frac{\alpha +\sqrt {\alpha ^2-4} }{2}=\lambda \, ,
\]
The number $\lambda$, the \textit{characteristic constant} of $(b_{i})$, 
allows one to express the lens sequence by a Binet-type formula
\[
  b_{n} = w\lambda ^n+\bar {w}\bar {\lambda }^n+\gamma 
\]
where
\[
  w = \frac{a-2b+c}{2(\alpha -2)}+\frac{c-a}{2(\alpha ^2-4)}
	\sqrt {\alpha ^2-4}, \quad \gamma =\frac{-\beta }{\alpha -2}\,,
\]
and where the bar denotes natural conjugation in the field 
$\mathbb{Q}(\sqrt{\alpha^2-4})$\,. The constant $\lambda$
is an example of a (quadratic) Pisot number, an algebraic integer, 
the powers of which approximate natural numbers. 
In particular, $b_{n} \approx w\lambda^n+\gamma $. 
Lens sequences can be expressed also as combinations of Chebyshev polynomials.

The most mysterious property of a lens sequence is that its terms may be 
formed by taking products of pairs of consecutive terms of another sequence. 
This ``underground'' sequence has an alternating recurrence rule, different 
for odd and even terms. Namely $b_{n}=f_{n-1}f_{n}$ , where:
\[
  f_{n}=\left\{ {{\begin{array}{*{20}c}
 kf_{n-1} -f_{n-2} \quad \hbox{if}\ n \ \hbox{is even} \hfill \\
 sf_{n-1} -f_{n-2} \quad \hbox{if}\ n \ \hbox{is odd}.\; \hfill \\
\end{array} }} \right.
\]
The constants of the sequence $(b_{i})$ may now be expressed as 
\[
  \begin{cases}
	\alpha = ks - 2\\
	\beta = sf_{0}^{2} + kf_{1}^{2} - ks\,f_{0}f_{1}\,.
   \end{cases}
\]
It follows that the integer lens sequences may be determined by four arbitrary 
integers 
\[
  ^{s}(f_{0}, f_{1})^{k}\, .
\]
Choosing for $f_{0}$ the term with smallest absolute 
value allows one to treat the above quadruple as a \textbf{symbol} that 
labels the corresponding lens sequence. 

The underground sequences automatically satisfy the following two recurrence 
formulas:
$$
\begin{aligned}
          (i)& \ f_{n+2}+f_{n-2} = \alpha f_n \\
         (ii)& \ f_{n+2}f_{n-1}-f_{n}f_{n+1} = -\beta \quad\hbox{ or }\quad 
              \det\left[\begin{matrix}
                          f_{n-1} & f_{n}\\
                          f_{n+1} & f_{n+2} 
                   \end{matrix}\right]
              =-\beta
\end{aligned}
$$
More precisely, the set of the sequences that are underground sequences for lens sequences 
coincides with the intersection 
$\Lambda\cap\Delta$, where $\Lambda$ denotes the space of sequences satisfying linear 
recurrence (i), and $\Delta$ denotes the set of sequences satisfying (ii).

An intriguing property holds --- the eigenvectors of matrices assembled from 
the terms of a lens sequence are vectors with entries from the corresponding 
underground sequence:
\[
  \left[ {{\begin{array}{*{20}c}
 {b_{n+1} } \hfill & {b_{n+2} }  \\
 {b_{n-1} } \hfill & {b_n }  \\
\end{array} }} \right]
\left[ {{\begin{array}{*{20}c}
 {f_{n+1} } \\
 {-f_{n-1} } \\
\end{array} }} \right] = \beta \left[ {{\begin{array}{*{20}c}
 {f_{n+1} } \\
 {-f_{n-1} } \\
\end{array} }} \right] .
\]
But the full meaning of the underground sequences remains to be 
understood.\\

%-------------------------------------------------------------------------
\section*{Appendix}

Below, we summarize the general formulas for each of the five types of symmetric integer lens sequences. Recall that $L$ = length of the lens, $R$ = radius of the lens circles,
$\delta$ = their relative distance, and $\lambda$ = characteristic constant.
Due to symmetry, the sums of the reciprocals are curtailed to one (right) 
tail of the sequence. 
\\
\\

%=================1 =================================
\hrule\vskip1ex\noindent
\begin{enumerate}
\item[1.] Seed: $[n,1,n]$ . \qquad 
Symbol: ${}^{n+1}(1,1)^{n+1}$. \qquad
Recurrence:  
$\alpha=(n+1)^2-2$, \   
$\beta=1-n^2$
\\[5pt]    
Geometry: \quad
$R=\frac{n+1}{n-1}$,  \quad 
$L=2\frac{n+3}{n-1}$,  \quad 
$\delta=\frac{4}{n+1}$.  \quad(Inner chain)
\\
Characteristic constant:\quad
$\lambda = \left(  \frac{{n+1} + \sqrt{(n+3)(n-1)}  }{2}\right)^2
         =\frac{ n^2 + 2n -1 +(n+1)\sqrt{(n+3)(n-1)}}{2}$
\\[5pt]
Binet: \quad
$b_k = \frac{ \lambda^k +\bar\lambda^k +n+1}{n+3}$
\\[5pt]
Sum of reciprocals: \quad
$\sum\limits_{k=0}^\infty 1/b_k = 1 + \frac{1}{n} +\ldots 
   =  \frac{1}{2}+\frac{1}{2}\sqrt{\frac{n+3}{n-1}} $
\\
\\
%=======================================================
\hrule\vskip1ex\noindent
\item[2.] Seed: $[n,1,1,n]$ . \qquad 
Symbol: ${}^{2}(1,1)^{n+1}$. \qquad
Recurrence:  
$\alpha=2n$ ,  \ 
$\beta=1-n$ .
\\[5pt]
Geometry: \quad
$R=2\frac{n+1}{n-1}$,  \quad 
$L=2\sqrt{\frac{n+1}{n-1}}$,  \quad 
$\delta=\frac{4}{\sqrt{2(n+1)}}$. \quad(Inner chain)  
\\ 
Characteristic constant:\quad
$\lambda = \left(  \frac{ \sqrt{2n+2} + \sqrt{2n-2} }{2}\right)^2
         =\frac{n +\sqrt{n^2-1}}{2}$
\\[5pt]
Binet: \quad
$b_k % = \frac{ (n+1+\sqrt{n^2-1})\lambda^k +(n+1-\sqrt{n^2-1})\bar\lambda^k +2(n+1)}{4(n+1)}
     = \frac{1}{4}\left(1+\sqrt{\frac{n-1}{n+1}}\right)\lambda^k 
      +\frac{1}{4}\left(1-\sqrt{\frac{n-1}{n+1}}\right)\lambda^{-k} +\frac{1}{2}$
\\[5pt]
Sum of reciprocals: \quad
$\sum\limits_{k=1}^\infty 1/b_k = 1 + \frac{1}{n} +\ldots
= \sqrt{\frac{n+1}{n-1}}$
\\
\\
%=======================================================
\hrule\vskip1ex\noindent
\item[3.] Seed: $[n,2,2,n]$ . \qquad 
Symbol: ${}^{1}(2,1)^{n+2}$. \qquad
Recurrence:  
$\alpha=n$ ,  \ 
$\beta=2-n$ .
\\[5pt]
Geometry: \quad
$R=\frac{n+2}{n-2}$,  \quad 
$L=2\sqrt{R}=2\frac{n+2}{n-2}$,  \quad 
$\delta=\frac{4}{\sqrt{n+2}}$. \quad(Inner chain)  
\\ 
Characteristic constant:\quad
$\lambda = \left(  \frac{ \sqrt{n+2} + \sqrt{(n-2)} }{2}\right)^2
         =\frac{n +\sqrt{n^2-4}}{2}$
\\[5pt]
Binet: \quad
$b_k %= \frac{ (n+2+\sqrt{n^2-4})\lambda^k +(n+2-\sqrt{n^2-4})\bar\lambda^k +2(n+2)}{2(n+2)}$
      = \frac{1}{2}\left(1+\sqrt{\frac{n-2}{n+2}}\right)\lambda^k 
       +\frac{1}{2}\left(1-\sqrt{\frac{n-2}{n+2}}\right)\lambda^{-k} +1$
\\ \\
Sum of reciprocals: \quad
$
\sum\limits_{k=1}^\infty 1/b_k = \frac{1}{2} + \frac{1}{n} +\ldots
= \frac{1}{2}\sqrt{\frac{n+2}{n-2}}$
\\
\\

%=======================================================
\hrule\vskip1ex\noindent
\item[4.] Seed: $[n,-1,n]$ . \qquad 
Symbol: ${}^{n-1}(1,1)^{n+1}$. \quad
Recurrence:  
$\alpha=(n-1)^2-2$, \ 
$\beta=n^2-1$.
\\[5pt]
Geometry: \quad
$R=-\frac{n-1}{n+1}$,  \quad 
$L=2\frac{n-3}{n+1}$,  \quad 
$\delta=\frac{4}{n-1}$. \quad(Outer chain)  
\\ \\
Characteristic constant:\quad
$\lambda = \left(  \frac{ {n-1} + \sqrt{(n-3)(n+1)}  }{2}\right)^2
         =\frac{n^2-2n-1 +(n-1)\sqrt{(n-3)(n+1)}}{2}$
\\[5pt]
Binet: \quad
$b_k = \frac{ \lambda^k +\bar\lambda^k -(n-1)}{n-3}$
\\[5pt]
Sum of reciprocals: \quad
$\sum\limits_{k=1}^\infty 1/b_k =  \frac{1}{n} +\ldots 
=  \frac{4}{n+1}$
\\
\\
%===============================================
\hrule\vskip1ex\noindent
\item[5.] Seed: $[0,1,n]$ . \qquad 
Symbol: ${}^{1}(1,1)^{n+1}$. \qquad
Recurrence:  
$\alpha = n-1$ , \   
$\beta = 1$.
\\[5pt]
Geometry: \quad
$R=-(n+1)$,  \quad 
$L=2\sqrt{(n+1)(n-3)}$,  \quad 
$\delta=\frac{4}{\sqrt{n+1}}$.  (Outer chain)
\\[5pt]
Characteristic constant:\quad
$\lambda = \left(  \frac{{n+1} + \sqrt{(n-3)}}{2}\right)^2
         =\frac{ n -1 +\sqrt{(n+1)(n-3)}}{2}$
\\[5pt]
Binet: \quad
$b_k %=\frac{ ((n-2)(n+1)+n\sqrt{(n+1)(n-3)})\lambda^k 
       %         +((n-2)(n+1)-n\sqrt{(n+1)(n-3)})\bar\lambda^k - 2(n+1)}{2(n+1)(n-3)}
     = \frac{1}{2}\left(\frac{n-2}{n-3}  + \frac{n}{\sqrt{(n+1)(n-3)}}\right)\lambda^k 
      +\frac{1}{2}\left(\frac{n-2}{n-3}  - \frac{n}{\sqrt{(n+1)(n-3)}}\right)\lambda^{-k} 
                +\frac{1}{n-3}$
\\[5pt]
Sum of reciprocals: \quad
$\sum\limits_{k=0}^\infty 1/b_k = 1 + \frac{1}{n} +\ldots
   =  \frac{n+1-\sqrt{(n+1)(n-3)}}{2} $

\end{enumerate}

%-------------------------------------------------------------------------
\section*{Acknowledgements}

I want to thank Philip Feinsilver, Alan Schoen, and the participants of 
the ``Apollonian Seminar" held at SIU for their interest and many helpful remarks.

%-------------------------------------------------------------------------

\noindent 2000 {\it Mathematics Subject Classification}: Primary 11B37 (recurrence) 11B39 (Fibonacci, Lucas, and generalizations) 11B83 (Special sequences and polynomials)\\

\noindent {\it Keywords}: recurrence, Binet, Apollonian window, circles, Chebyshev, self-generating sequence.

\bigskip
\hrule
\bigskip

\noindent (Concerned with sequences 
\seqnum{A000032},    %Lucas
\seqnum{A000045},    %Fibonacci
\seqnum{A000079}, 
\seqnum{A000217}, 
\seqnum{A000466},
\seqnum{A001654},   %* (exact) 
\seqnum{A001834},   %[1,5;4 -1]  (undergr Ex7)
\seqnum{A001835},   % 1 1 3 11 41 (undergr Ex2)
\seqnum{A002878},   %odd Lucas nrs 
\seqnum{A005013},   %alt Fib/Luc 1 1 4 3 11 8 Ex 16
\seqnum{A005246},   %undergr ex 9
\seqnum{A005247},  %altern Luc/Fibb  1 3 2 6 5 18 13
\seqnum{A005408},   %odd nrs
\seqnum{A011900}, 
\seqnum{A011922},
\seqnum{A021913}, 
\seqnum{A026741},   %odd,even/2
\seqnum{A027941}, 
\seqnum{A032908}, 
\seqnum{A032908}, 
\seqnum{A054318}, 
\seqnum{A061278}, 
\seqnum{A064170}, 
\seqnum{A075269},   %* (absolute values)
\seqnum{A081078},
\seqnum{A084158},   %* (absolute values)
\seqnum{A084159},   %* (absolute values)
\seqnum{A099025},   %* (absolute values)
\seqnum{A101265},  
\seqnum{A101879} and
\seqnum{A109437}.)  %*(absolute values)

\bigskip
\hrule
\bigskip

\end{document}